\documentclass[11pt, reqno]{amsart}

\usepackage{cite}
\usepackage[english]{babel}
\usepackage[T1]{fontenc}
\usepackage{csquotes}
\usepackage{amsmath, amssymb, amsthm}
\usepackage{enumitem}
\usepackage{aligned-overset}
\usepackage{xifthen}
\usepackage{color}
\definecolor{MyLinkColor}{rgb}{0,0,0.4}

\newlength{\heightofnormal}
\newlength{\heightofinput}
\newcommand{\set}[1]{\left\{#1\right\}}

\newcommand{\p}[1]{
    \settoheight{\heightofinput}{$#1$}
    \ifdim\heightofinput > \heightofnormal
        \left(#1\right)
    \else
        (#1)
    \fi
}
\newcommand{\sqp}[1]{
    \settoheight{\heightofinput}{$#1$}
    \ifdim\heightofinput > \heightofnormal
        \left[#1\right]
    \else
        [#1]
    \fi
}
\newcommand{\norm}[1]{
	\settoheight{\heightofinput}{$#1$}
    \ifdim\heightofinput > \heightofnormal
        \left\lVert #1\right\rVert
    \else
        \lVert #1\rVert
    \fi
}
\newcommand{\abs}[1]{
	\settoheight{\heightofinput}{$#1$}
    \ifdim\heightofinput > \heightofnormal
        \left\lvert #1\right\rvert
    \else
        |#1|
    \fi
}

\newcommand{\dx}[1]{\,\mathrm{d}#1}
\newcommand{\dfx}[2]{\frac{\mathrm{d}#1}{\mathrm{d}#2}}
\newcommand{\Lp}[2]{\mathrm{L}^{#1}(#2)}
\newcommand{\Hp}[2]{\mathrm{H}^{#1}(#2)}
\newcommand{\Wp}[2]{\mathrm{W}^{#1,2}(#2)}
\newcommand{\Wpp}[3]{\mathrm{W}^{#1,#2}(#3)}
\newcommand{\nWpp}[2]{\mathrm{W}^{#1,#2}}
\newcommand{\nWp}[1]{\mathrm{W}^{#1,2}}

\newcommand{\nHp}[1]{\mathrm{H}^{#1}}
\newcommand{\hHp}[2]{\widehat{\mathrm{H}}\vphantom{\mathrm{H}}^{#1}(#2)}

\newcommand{\PV}{\mathrm{PV}\!}
\newcommand{\Rel}{\mathrm{Re}\,}

\newcommand{\Aa}{\mathbb{A}}
\newcommand{\ve}{\varepsilon}
\newcommand{\RR}{\mathbb{R}}
\newcommand{\CC}{\mathbb{C}}
\newcommand{\NN}{\mathbb{N}}
\newcommand{\ZZ}{\mathbb{Z}}

\newcommand{\Ss}{\mathbb{S}}

\newcommand{\bfa}{\mathbf{a}}
\newcommand{\bfb}{\mathbf{b}}
\newcommand{\bfc}{\mathbf{c}}
\newcommand{\bfd}{\mathbf{d}}

\newcommand{\mcD}{\mathcal{D}}

\newcommand{\mcH}{\mathcal{H}}
\newcommand{\mcL}{\mathcal{L}}

\newcommand{\mcP}{\mathcal{P}}

\newcommand{\mcU}{\mathcal{U}}

\newcommand{\rmC}{\mathrm{C}}

\newcommand{\sfv}{\mathsf{v}}
\newcommand{\sfq}{\mathsf{q}}
\newcommand{\wt}{\widetilde}
\newcommand{\tnu}{{\tilde\nu}}
\newcommand{\tkappa}{{\tilde\kappa}}

\newcommand{\ts}[1]{t_{[#1]}}
\newcommand{\tss}{\ts{s}}
\newcommand{\T}[1]{T_{[#1]}}
\newcommand{\tx}{\ts{x_1}}
\newcommand{\Tx}{\T{x_2}}
\newcommand{\Tf}[2]{T_{[#1]}#2}
\newcommand{\dg}[2]{\delta_{[#1]}#2}
\newcommand{\dxsf}{\dg{\xi,s}{f}}
\newcommand{\Txsf}{\Tf{\xi,s}{f}}
\newcommand{\id}{\mathrm{id}}

\DeclareMathOperator*{\supp}{supp}
\DeclareMathOperator*{\vdiv}{div}

\newcommand{\zd}{.\kern-\nulldelimiterspace}

\newtheorem{Theorem}{Theorem}[section]
\newtheorem{Proposition}[Theorem]{Proposition}
\newtheorem{Lemma}[Theorem]{Lemma}
\newtheorem{Corollary}[Theorem]{Corollary}
\theoremstyle{remark} 
\newtheorem{Remark}[Theorem]{Remark}

\numberwithin{equation}{section}
\usepackage[colorlinks=true,linkcolor=MyLinkColor,citecolor=MyLinkColor]{hyperref}

\begin{document}
\title[Well-posedness and stability for the two-phase periodic quasistationary Stokes flow]{Well-posedness and stability for the two-phase periodic quasistationary Stokes flow}
\thanks{Partially supported by DFG Research Training Group~2339 ``Interfaces, Complex Structures, and Singular Limits in Continuum Mechanics - Analysis and Numerics''}
\author{Daniel B\"ohme}
\author{Bogdan-Vasile Matioc}
\address{Fakult\"at f\"ur Mathematik, Universit\"at Regensburg \\ D--93040 Regensburg, Deutschland}
\email{daniel.boehme@ur.de}
\email{bogdan.matioc@ur.de}

\begin{abstract}
 The two-phase horizontally periodic quasistationary Stokes flow in $\mathbb{R}^2$, describing the motion of two immiscible fluids with equal viscosities that are separated 
by a sharp interface, which is parameterized as the graph of a function $f=f(t)$,  is considered in the general case when both gravity and surface tension effects are included.
Using potential theory, the moving boundary problem is formulated as a fully nonlinear and nonlocal parabolic problem for the function~$f$.  
Based on abstract parabolic theory, it is proven that the problem is well-posed in all subcritical spaces~$\mathrm{H}^r(\mathbb{S})$, $r\in(3/2,2)$.
Moreover, the stability properties of the flat equilibria are analyzed in dependence on the physical properties of the fluids.
\end{abstract}


\maketitle

\pagestyle{myheadings}
\markboth{\sc{D. B\"ohme \& B.-V.~Matioc}}{\sc{Well-posedness and stability for the two-phase periodic quasistationary Stokes flow}}

\section{Introduction}\label{Sec:Introduction}
		We consider the two-phase horizontally periodic quasistationary Stokes flow driven by surface tension and gravity effects, which is modeled by the system
	\begin{subequations}\label{eq:STOKES}
	\begin{equation}\label{eq:Stokes}
		\left.
		\arraycolsep=1.4pt
		\begin{array}{rclll}
		\mu\Delta v^\pm-\nabla q^\pm&=&0&\mbox{in $\Omega^\pm(t)$,}\\
		\vdiv v^\pm&=&0&\mbox{in $\Omega^\pm(t)$,}\\{}
		[v]&=&0&\mbox{on $\Gamma(t)$,}\\{}
		[T_\mu (v,q)]\tnu&=&(\Theta x_2-\sigma\tkappa) \tnu&\mbox{on $\Gamma(t)$,}\\
		(v^\pm,q^\pm)(x)&\to&(\pm 	c_{1,\Gamma},0,\pm 	c_{2,\Gamma})&\mbox{for $x_2\to\pm\infty$,}\\
		V_n&=&v\cdot\tnu&\mbox{on $\Gamma(t)$}
		\end{array}\right\}
	\end{equation}
	for $t>0$.
	We assume that $\Gamma(t)$ is the graph of a function $f(t)$ that separates the   two fluid domains  
	\begin{equation*}
		\Omega^\pm (t):= \{x=(x_1 ,x_2)\in\Ss\times\RR: x_2 \gtrless f(t,x_1)\},\qquad t>0.
	\end{equation*}
	We denote by $\Ss:=\RR/2\pi\ZZ$ the unit circle, functions that depend on the real variable $\xi\in\Ss$ being $2\pi$-periodic with respect to it. 
	In particular, the unknown $(f,v^\pm, q^\pm)$ is assumed to be~$2\pi$-periodic with respect to the horizontal variable $x_{1}$.
	 At time $t=0$ we impose the initial condition
	\begin{equation}\label{eq:ic}
		f(0)=f_0.
	\end{equation}	
Moreover, the constants $	c_{1,\Gamma}$ and $	c_{2,\Gamma}$ evolve over time and are identified by $f=f(t)$ and the  other constants in \eqref{eq:Stokes}  according to 
	\begin{equation}\label{eq:c_inf}
	c_{1,\Gamma}:=-\frac{\sigma}{2\mu}\Big\langle\frac{f'}{(1+f'^2)^{1/2}}\Big\rangle\qquad\text{and}\qquad 	c_{2,\Gamma}:=-\frac{\Theta}{2}\langle f\rangle,\qquad t>0,
	\end{equation}
	\end{subequations}
	 where 
	\[
	\langle h \rangle :=\frac{1}{2\pi}\int_{-\pi}^\pi h(s) \dx{s}
	\]
	 denotes the integral mean of  an  integrable function $h:\Ss\to\RR$ and 
	\begin{equation}\label{eq:defTheta}
	\Theta:=g(\rho^--\rho^+)\in\RR.
	\end{equation}
	 The equation set \eqref{eq:STOKES} describes the dynamics of two incompressible and immiscible Newtonian fluids with equal viscosities,
	 the positive constants $\sigma$ and $\mu$ representing the surface tension coefficient at the interface and the viscosity of the fluids, respectively. 
	 The constant~$g{\geq 0}$ is the   Earth's gravity and $\rho^\pm$ stands for the density of the fluid occupying $\Omega^\pm$. 
	Moreover,~$\tnu$ is the unit exterior normal to $\partial\Omega^-$ and  $\tkappa$ is the curvature of the moving interface.
We further denote by~${v^\pm =v^\pm (t):\Omega^\pm (t)\to\RR^2}$  the velocity field  in the fluid domain~$\Omega^\pm (t)$  and~$q^\pm =q^\pm (t): \Omega^\pm(t)\to \RR$ is defined by
\[
q^\pm(t,x)=p^\pm(t,x)+g\rho^\pm x_2,\qquad x=(x_1,x_2)\in \Omega^\pm (t),
\] 
where $p^\pm=p^\pm(t)$ is the pressure in $\Omega^\pm (t)$.
The stress tensor $T_{\mu}(v^\pm,q^\pm)$  is given by
	\begin{equation}\label{eq:stress}
		T_{\mu}(v^\pm,q^\pm):= -q^\pm I_2 +\mu \big(\nabla v^\pm +(\nabla v^\pm)^\top\big),\qquad (\nabla v^\pm)_{ij} := \partial_j v^\pm_i,\quad i,\,j=1,\, 2,
	\end{equation}
	while $[v]$ and  $[T_{\mu}(v,q)]$ denote the jump of the velocity and of the stress tensor across the moving interface, respectively, as defined in \eqref{eq:defjump} below. 
	Finally, $V_n$ is the normal velocity of the interface $\Gamma(t)$, $x\cdot y$ denotes the Euclidean scalar product of two vectors $x,\,y\in\RR^n$, and~$I_2\in\RR^{2\times 2}$ is the identity matrix.

First studies of the quasistationary Stokes flow investigated the setting of  a single fluid phase which occupies a sufficiently smooth domain~$\Omega(t)\subset\RR^d$, $d\geq 2$. 
In \cite{PG97}, the authors have established the well-posedness of the problem for  initial data which are close to a smooth and strictly star-shaped domain, together with the exponential stability of balls.
Subsequently, the exponential stability of balls has been proven in \cite{Fr02a, Fr02b} in space dimension~${d\in\{2,\, 3\}}$ by a different power series argument.
Furthermore, in  the confined  setting, the  quasistationary Stokes flow is the singular limit of the  Navier-Stokes problem  when the Reynolds number vanishes, cf. \cite{S99, Solo99}. 

The two-phase  quasistationary Stokes flow in a bounded geometry, with one of the fluid phases surrounding the other one and  with possible phase transitions, has been considered in arbitrary  space dimension 
 $d\geq2$ in the monograph~\cite{PS16} on the basis of maximal regularity in weighted $\mathrm{L}^p$-spaces.

In the absence of gravity effects,   the  nonperiodic version of the  quasistationary Stokes flow~\eqref{eq:STOKES} -- with equal and general viscosities -- has been investigated   quite recently in~\cite{Matioc.2021, MP2022},
 the local well-posedness  property being provided  in $\Hp{r}{\RR}$, with~${r\in(3/2,2)}$  arbitrarily close to the critical exponent~${r=3/2}$, cf. \cite[Remark~1.2]{MP2022}.
 Moreover, as shown in \cite{MP2023}, the unconfined one-phase flow is the singular limit of the two-phase problem when the viscosity of one of the fluids vanishes.

In the absence of surface tension effects, that is when $\sigma=0$, the problem \eqref{eq:STOKES}  has been  analyzed in \cite{GGS23x, GGS24x}.
If the interface between the fluids is parameterized as an arbitrary curve,  the problem can be reformulated as an ODE and local well-posedness is established by using Picard's theorem \cite{GGS23x}.
In \cite{GGS24x}, the authors   showed  that the problem is actually globally well-posed.
Global existence results of solutions in the graph case are provided in \cite{GGS23x} for initial data that are small in $\Hp{3}{\Ss}$ or in certain Wiener algebras in the stable regime when~$\Theta>0$.
In fact, the solutions to \eqref{eq:STOKES} in the case $\sigma=0$ also  solve the   transport Stokes system, see  \cite{H18, Me19, MS22, Leb21x, Gr23}, 
which is a model for the  settling process of a cluster of rigid particles in a viscous fluid. 

Lastly, we mention  the related Peskin problem  which models the evolution of an  elastic  string  (or membrane) immersed in a viscous fluid. 
In this context, the equations governing the motion in the fluid match those in equation \eqref{eq:STOKES} (with~$g=0$), while the dynamics of the elastic string are 
described using Lagrangian coordinates, see the  very recent research in~\cite{LT19, MRS19, CS24, CN23, GMS23, GGS23, ES23xx, EKMS23xx}.

In this paper, we consider the horizontally periodic quasistationary Stokes flow \eqref{eq:STOKES} with both gravity (neglected if $g=0$) and surface tension effects incorporated -- 
which could not be achieved in the nonperiodic case~\cite{Matioc.2021, MP2022} and was neither investigated in \cite{GGS23x, GGS24x} -- in Sobolev spaces
 $\Hp{r}{\Ss}$ with~${r\in(3/2,2)}$ arbitrary (again $r=3/2$ is the critical exponent). 
	A striking difference to the nonperiodic case  ~\cite{Matioc.2021, MP2022} is the far field boundary condition~\eqref{eq:Stokes}$_5$ for~$(v,q)$.
	While in the nonperiodic case~$(v,q) $ vanishes at infinity,  under the periodicity assumption~$(v,q)$ converges, at each fixed time $t$,
	 towards a constant vector   explicitly determined by~$f(t)$, cf.~\eqref{eq:c_inf}. 
	In particular,  for $x_2\to\pm\infty$, the velocity is asymptotically  horizontal, but the asymptotic profiles at $\pm\infty$ are opposite to each other. 
	Moreover, for~$x_2\to\pm \infty $, the pressure may deviate from the hydrostatic pressure by some constant which is determined by $f(t)$ and which has opposite sign at $\pm\infty$.
	
Our strategy to solve \eqref{eq:STOKES} is to  prove that $(v^\pm,q^\pm)$ is determined at each time instant~$t>0$ by~$f(t)$,  provided~$f(t)\in \Hp{3}{\Ss}$, see Remark~\ref{R:FTP} and Theorem~\ref{Thm:FT_unique}.
In this way we may reformulate \eqref{eq:STOKES} as a  fully nonlinear and nonlocal problem for~$f$, see \eqref{eq:ev_eq}, with nonlinearities expressed by (singular) integral operators involving $f$, 
which are well-defined when merely~${f \in \Hp{r}{\Ss}}$ with ${r>3/2}$. 
 The fully nonlinear character of \eqref{eq:ev_eq} is due to the fact that the phase space~$\Hp{r}{\Ss}$ can be chosen arbitrarily close to the critical space, meaning that the curvature has to be interpreted as a distribution. 
The situation is different  in
 \cite{Fr02a, Fr02b, PS16, S99, Solo99, PG97} where the interface is at least of class ${\rm C}^2$ and, due to quasilinearity of the curvature operator, the Stokes problem~\eqref{eq:STOKES}
 may be formulated as a quasilinear evolution problem.
Using results on certain families of singular integral operators   from Appendix~\ref{Sec:A} and Appendix~\ref{Sec:B} -- which might be  of interest also in the context of other evolution problems (such as the Hele-Shaw, Muskat, or Mullins-Sekerka problems) --
 we then prove that the new formulation~\eqref{eq:ev_eq} is of parabolic type and provide its well-posedness by using abstract parabolic theory from~\cite{Lunardi.1995}, cf. Theorem~\ref{Thm:main} below.
 Additionally, we  show  a parabolic smoothing property  for~\eqref{eq:ev_eq}, which justifies the assumption~${f(t)\in \Hp{3}{\Ss}}$  in Theorem~\ref{Thm:FT_unique}.
 	
 \pagebreak
 
Our  first main  result  is the following  theorem.

	\begin{Theorem}\label{Thm:main}
		Let $\Theta\in\RR$, $\mu,\, \sigma\in(0,\infty)$, and $r\in(3/2,2)$.  
        \begin{enumerate}[label=\textup{(\roman*)}]
            \item \textup{(Well-posedness)} Given $f_0\in\Hp{r}{\Ss}$, there exists a unique maximal solution $(f,v^\pm\!,q^\pm)$ to~\eqref{eq:STOKES} such that 
            \begin{align*}
                f=f(\cdot, f_0)\in\rmC([0,T_+),\Hp{r}{\Ss})\cap\rmC^1([0,T_+),\Hp{r-1}{\Ss})
            \end{align*}
            and 
            \begin{equation*} 
              \left.          
            \begin{array}{lll}
            &f(t)\in \Hp{3}{\Ss}, \\
                & v^{\pm}(t)\in \mathrm{C}^{2}(\Omega^\pm(t),\RR^2)\cap \mathrm{C}^1(\overline{\Omega^\pm(t)},\RR^2),\\
                & q^\pm(t) \in \mathrm{C}^1(\Omega^\pm(t))\cap \mathrm{C}(\overline{\Omega^\pm(t)})
            \end{array}
            \right\}\qquad\text{for all $t\in(0,T_+)$,}
            \end{equation*}
            where $T_+=T_+(f_0)\in (0,\infty]$.\\[-2ex]
            \item \textup{(Parabolic smoothing)} We have $[(t,\xi)\mapsto f(t,\xi)]\in \rmC^\infty((0,T_+)\times\Ss,\RR).$\\[-2ex]
            \item \textup{(Global existence)} The solution is global, that is  $T_+(f_0)=\infty,$ if for each $T>0$ we have 
            \begin{equation*}
                \sup_{[0,T]\cap[0,T_+(f_0))}\norm{f(t)}_{\nHp{r}}<\infty.
            \end{equation*} 
        \end{enumerate}
	\end{Theorem}
	
 Our second main objective is to study the stability properties of the solutions with a flat interface, which are all equilibria to  \eqref{eq:STOKES}. 
	Indeed,  if $(f,v^\pm,q^\pm)$ is a solution to \eqref{eq:STOKES}, then, for each constant $c\in\RR$, the tuple~${(f+c,\tilde v^\pm,\tilde q^\pm)}$ defined by
	\[
\tilde v^\pm (t,x)=	  v^\pm(t,x-(0,c))\quad\text{and}\quad \tilde q^\pm (t,x)=	  q^\pm(t,x-(0,c))\mp\frac{c\Theta}{2},\qquad x_2\neq f(t,x_1)+c,
	\]
	is again a solution to \eqref{eq:STOKES} (with initial data $f_0+c$ and the same maximal existence time).
	Moreover, the integral mean $\langle f\rangle$ is preserved by the flow since, by \eqref{eq:Stokes}$_2$, \eqref{eq:Stokes}$_5$, and~\eqref{eq:Stokes}$_6$,  
	\begin{equation}\label{eq:pres}
	\frac{{\rm d}\langle f\rangle}{{\rm d}t}(t)=\int_{\Gamma(t)}v(t)\cdot\tnu(t)\,\dx{\sigma}=\int_{\Omega^\pm(t)}\vdiv v^\pm(t)\,\dx{x}=0,\qquad t\in[0,T_+).
	\end{equation}
	Since it is straightforward to verify that $(f,v,q)=0$ is a stationary solution to \eqref{eq:STOKES}, it follows that $(c,0,\mp c\Theta/2)$ is a stationary solution to \eqref{eq:STOKES} for each $c\in\RR.$
	 This observation together with our second main result in Theorem~\ref{Thm:stability} shows, on the one hand,  that if 
	 \begin{equation}\label{eq:Stab}
\sigma+\Theta >0,
	 \end{equation}
	 solutions corresponding to   initial data $f_0 \in\Hp{r}{\Ss}$ which are close to a constant  exist globally and $f(t)$ converges exponentially fast  towards the integral mean of $f_0.$
	 On the other hand, if~\eqref{eq:Stab} holds with reverse inequality, that is $\sigma+\Theta <0$, then the constant solutions are (nonlinearly) unstable.

	\begin{Theorem}[Exponential stability/Instability]\label{Thm:stability}\,
	\begin{itemize}
	\item[{\rm (i)}] Assume~\eqref{eq:Stab} and let $\vartheta_0$ denote the positive constant
	\begin{equation}\label{theta0}
\vartheta_0:=
\left\{
\begin{array}{lll}
\cfrac{\sigma+\Theta}{4\mu}&,& \sigma\geq \Theta,\\[2ex]
\cfrac{\sqrt{\sigma\Theta}}{2\mu}&,& \sigma< \Theta.
\end{array}
\right.
	\end{equation}
	Then, given $\vartheta\in [0, \vartheta_0)$, there exist  constants $\delta>0$ and~$M>0$, such that for any~${f_0\in\Hp{r}{\Ss}}$ satisfying 
		\begin{equation*}
			\norm{f_0}_{\nHp{r}}<\delta \qquad \text{and}\qquad \langle f_0\rangle=0,
		\end{equation*}				
		we have $T_+(f_0)=\infty$ and
		\begin{equation*}
			\norm{f(t)}_{\nHp{r}}+\norm{\frac{\dx{f}}{\dx{t}}(t)}_{\nHp{r-1}}\leq M e^{-\vartheta t}\norm{f_0}_{\nHp{r}}\qquad \text{for all  $ t\geq 0$.}
		\end{equation*}
		\item[{\rm (ii)}] If $\sigma+\Theta<0$, then the zero solution is unstable.
	\end{itemize}
	\end{Theorem}
	
Outline: In Section~\ref{Sec:FTP} and Appendix~\ref{Sec:C} we  solve the fixed time problem~\eqref{eq:Stokes}$_1$-\eqref{eq:Stokes}$_5$ with a general right-hand side in \eqref{eq:Stokes}$_4$ and \eqref{eq:Stokes}$_5$.
Then, in Section~\ref{Sec:31}, we introduce two classes of (singular) integral operators,   studied in  Appendix~\ref{Sec:A} and Appendix~\ref{Sec:B}, which enable us to reformulate in 
Section~\ref{Sec:32} the flow \eqref{eq:STOKES} as a nonlinear and nonlocal evolution problem for~$f$.
In Section~\ref{Sec:33}-Section~\ref{Sec:34} it is shown that the problem is of parabolic type,  the main results being established in Section~\ref{Sec:35}.

\section{The fixed time problem}\label{Sec:FTP}

	In this section we address the solvability of the  boundary value problem~\eqref{eq:Stokes}$_1$-\eqref{eq:Stokes}$_5$  at a fixed time instant~${t>0}$, under the assumption that $f=f(t)$ is sufficiently regular
	 and with a general right-hand side in \eqref{eq:Stokes}$_4$ and \eqref{eq:Stokes}$_5$.
	More precisely, we fix $f\in\Hp{3}{\Ss}$ and consider the boundary value problem
	\begin{equation}\label{eq:refStokes}
		\left.
		\begin{array}{rclll}
		\mu\Delta v^\pm-\nabla q^\pm&=&0&\mbox{in $\Omega^\pm$,}\\
		\vdiv v^\pm&=&0&\mbox{in $\Omega^\pm$,}\\{}
		[v]&=&0&\mbox{on $\Gamma$,}\\{}
		[T_{\mu}(v,q)]\tnu&=&(\omega^{-1} G)\circ\Xi^{-1}&\mbox{on $\Gamma$,}\\
		(v^\pm,q^\pm)(x)&\to&(\pm c_1,0,\pm c_2)&\mbox{for $x_2\to\pm\infty$,}\\
		\end{array}\right\}
	\end{equation}	
	where  $G:=(G_1,G_2)\in\Hp{1}{\Ss}^2$ satisfies~$\langle G_1\rangle=0$
	and the constants $c_1,\, c_2\in\RR$ are arbitrary.
	The domains $\Omega^\pm$ and their common boundary $\Gamma$ are defined by
	\begin{equation*}
		\Omega^\pm:=\{x=(x_1 ,x_2)\in\Ss\times\RR: x_2\gtrless f(x_1)\},\quad \Gamma:= \{(\xi, f(\xi))\in\Ss\times\RR: \xi\in\Ss\}.
	\end{equation*}
Note that  $\Xi:= \Xi_f:= (\text{id}_\Ss, f)$  is a diffeomorphism that maps the $\Ss$  onto $\Gamma$. 
Further, $\nu$ and~$\tau$ are the componentwise pull-back under $\Xi$ of the unit normal $\tilde\nu$  on $\Gamma$ exterior to $\Omega^{-}$ and of the tangent $\tilde \tau $,  that is
	\begin{equation}\label{nutauzomega}
		\nu:=\nu(f) :=\omega^{-1}(-f', 1)^\top, \quad \tau:=\tau(f) :=\omega^{-1}(1,f')^\top,\quad \omega:=\omega(f):= (1+f'^2)^{1/2}.
	\end{equation}
	For any functions $z^\pm$ defined on $\Omega^\pm$, respectively, and having limits at some $(\xi,f(\xi))\in\Gamma$, we will write
	\begin{equation}\label{eq:defjump}
		[z](\xi,f(\xi)):= \lim_{\Omega^{+}\ni x\to(\xi,f(\xi))}z^{+}(x)-\lim_{\Omega^{-}\ni x\to(\xi,f(\xi))}z^{-}(x).
	\end{equation}
	\begin{Remark}\label{R:FTP} We note that in the particular case  when
	\begin{equation}\label{defHG}
 G:=G(f):=\Theta(- ff', f)-\sigma(\omega^{-1}-1,\omega^{-1} f')', 
	\end{equation}
	 we have $G\in \Hp{1}{\Ss}^2 $ and  
	\[
	(\omega^{-1} G)\circ\Xi^{-1}=\big(\Theta x_2-\sigma\wt \kappa\big) \wt \nu.
	\]
Consequently, the right-hand sides of \eqref{eq:refStokes}$_4$ and  of \eqref{eq:Stokes}$_4$ coincide in this case.
	\end{Remark}
	
	Before stating our result on the solvability of  \eqref{eq:refStokes} we introduce some further notation.
	To start, we define for  $0\neq x=(x_1,x_2)\in\Ss\times\RR$
	\begin{equation}\label{eq:GPi}
		G_\pi(x):=-\frac{1}{4\pi}\ln\left(\frac{\ts{x_1}^2+\T{x_2}^2}{(1+\ts{x_1}^2)(1-\T{x_2}^2)}\right)
		=-\frac{1}{4\pi}\ln \Big(\sin^2\Big(\frac{x_1}{2}\Big)+\sinh^2\Big(\frac{x_2}{2}\Big)\Big),
	\end{equation}
	which is the fundamental solution to the $x_1$-periodic Laplace equation,  that is, $G_\pi$ solves the equation $-\Delta G_{\pi}=\delta_{0}$ in $\mcD'(\Ss\times\RR)$.
	We use the shorthand notation
	\begin{equation}\label{eq:notation1}
		t_{[x_1]}:=\tan\!\left(\frac{x_1}{2}\right),\quad x_1\in\RR\setminus(\pi+2\pi\mathbb{Z}),\qquad\text{and}\qquad T_{[x_2]}:=\tanh\!\left(\frac{x_2}{2}\right),\quad x_2\in\RR.
	\end{equation}
	The $x_1$-periodic Stokeslet $(\mcU,\mcP)$, with $\mcU:=(\mcU^1,\mcU^2)$ and $\mcP:=(\mcP^1,\mcP^2)$,  is defined by
	\begin{equation}\label{eq:UP}
	\begin{aligned}
		&\mcU^1(x)=-\frac{1}{2\pi}\big(G_\pi(x)+x_2\partial_2 G_\pi(x),-x_2 \partial_1 G_\pi(x)\big),\qquad &&\mcP^1(x)=\partial_1 G_\pi (x),\\
		&\mcU^2(x)=-\frac{1}{2\pi}\big(-x_2\partial_1 G_\pi(x),G_\pi(x)-x_2\partial_2 G_\pi(x)\big),\qquad &&\mcP^2(x)=\partial_2 G_\pi (x)
	\end{aligned}
	\end{equation}
 for $x=(x_1,x_2)\in(\Ss\times\RR)\setminus\{0\},$ and we  may reexpress \eqref{eq:UP} as follows
	\begin{equation}\label{eq:UPdef}
	\begin{gathered}
		\mcU (x)=\frac{1}{8\pi}\left(\ln\Bigg(\frac{\ts{x_1}^2+\T{x_2}^2}{(1+\ts{x_1}^2)(1-\T{x_2}^2)}\Bigg)I_2-x_2
		\begin{pmatrix}
			-\frac{(1+\ts{x_1}^2)\T{x_2}}{\ts{x_1}^2+\T{x_2}^2} & \frac{\ts{x_1}(1-\T{x_2}^2)}{\ts{x_1}^2+\T{x_2}^2}\\[2ex]
			\frac{\ts{x_1}(1-\T{x_2}^2)}{\ts{x_1}^2+\T{x_2}^2} & \frac{(1+\ts{x_1}^2)\T{x_2}}{\ts{x_1}^2+\T{x_2}^2}
		\end{pmatrix}
		\right),\\
		\mcP^1(x)=-\frac{1}{4\pi}\frac{\ts{x_1}(1-\T{x_2}^2)}{\ts{x_1}^2+\T{x_2}^2},\qquad\mcP^2(x)=-\frac{1}{4\pi}\frac{(1+\ts{x_1}^2)\T{x_2}}{\ts{x_1}^2+\T{x_2}^2}.
	\end{gathered}
	\end{equation}
	Using \eqref{eq:UP}, it is straightforward to prove that $(\mcU^k,\mcP^k)$, $k=1,\,2$, are fundamental solutions to the   Stokes equations in the sense that
	\begin{equation}\label{eq:HSE}
	\left.
	\begin{array}{rclll}
		\Delta\mcU^{k}-\nabla\mcP^k&=&\delta_0 e_k ,\\[1ex]
		\vdiv \mcU^k &=& 0
	\end{array}	
	\right\}\qquad		\mbox{in $\mcD'(\Ss\times\RR)$},
	\end{equation}		
	 where $e_1:=(1,0)$ and $e_2:=(0,1)$. 
	 In particular, they solve the  Stokes equations \eqref{eq:HSE} pointwise in $(\Ss\times\RR )\setminus\{0\}.$
	For the derivation of the $x_1$-periodic Stokeslet $(\mcU,\mcP)$,  we refer to the recent paper  \cite{GGS23x } (see also \cite{Bohme.2026} for an alternative derivation).
	
	In Theorem~\ref{Thm:FT_unique} below and further on we sum over indices appearing twice in a product.

	\begin{Theorem}\label{Thm:FT_unique}
		Given $f\in\Hp{3}{\Ss}$ and $G\in\Hp{1}{\Ss}^2$ with $\langle G_1\rangle=0$, the  boundary value  problem~\eqref{eq:refStokes} 
		has a solution $(v^\pm,q^\pm)$ such that 
		\begin{equation}\label{eq:regcla}
		v^{\pm}\in \mathrm{C}^{2}(\Omega^\pm,\RR^2)\cap \mathrm{C}^1(\overline{\Omega^\pm},\RR^2)\qquad\text{ and}\qquad  q^\pm \in \mathrm{C}^1(\Omega^\pm)\cap \mathrm{C}(\overline{\Omega^\pm})
		\end{equation}
		 if and only if the constants $c_1,\, c_2$ in \eqref{eq:refStokes}$_5$ are given by
		\begin{equation}\label{eq:c1c2}
			c_1 = - \frac{\langle fG_1\rangle}{2\mu}\qquad\text{and}\qquad c_2= -\frac{\langle G_2\rangle}{2}.
		\end{equation}
		If $c_1,\, c_2$ are defined by \eqref{eq:c1c2}, then the solution $(v^\pm,q^\pm)$  is unique and is given by the formula
		\begin{equation}\label{eq:vq}
		v^\pm:=v_G^\pm+\Big(0,\frac{\langle G_2\rangle\ln 4}{4\mu}\Big)\qquad\text{and}\qquad q^\pm:=q^\pm_G,
	\end{equation}
 where, given $x\in\Omega^\pm$,
	\begin{equation}\label{eq:vqdef}
	\begin{aligned}
		&v_G^\pm(x):=\frac{1}{\mu} \int_{-\pi}^\pi  \mcU^k (x-(s,f(s))) G_k(s)\dx{s},\quad && q_G^\pm(x):= \int_{-\pi}^\pi  \mcP^{k}(x-(s,f(s)))G_k(s)\dx{s}.
	\end{aligned}
	\end{equation}
	\end{Theorem}
	\begin{proof} We devise the proof into  several parts.\medskip

	\noindent{\bf Uniqueness.} For the uniqueness statement, we  need to show that if  $(\sfv^\pm,\sfq^\pm)$ satisfies \eqref{eq:regcla} and solves the boundary value problem
		\begin{equation}\label{eq:Stokes_zero}
		\left.
		\begin{array}{rclll}
		\mu\Delta \sfv^\pm-\nabla \sfq^\pm&=&0&\mbox{in $\Omega^\pm$,}\\
		\vdiv \sfv^\pm&=&0&\mbox{in $\Omega^\pm$,}\\{}
		[\sfv]&=&0&\mbox{on $\Gamma$,}\\{}
		[T_\mu(\sfv,\sfq)]\tnu&=&0&\mbox{on $\Gamma$,}\\
		(\sfv^\pm,\sfq^\pm)(x)&\to&(\pm {\sf c}_1,0,\pm {\sf c}_2)&\mbox{for $x_2\to\pm\infty$ }
		\end{array}\right\}
		\end{equation}
		for some $({\sf c_1}, {\sf c_2})\in\RR^2,$ then actually $(\sfv^\pm,\sfq^\pm)=(0,0)$ in $\Omega^\pm$ and ${\sf c_1}={\sf c_2}=0$.  
		We first note, in view of~\eqref{eq:Stokes_zero}$_2$, that 
		\begin{equation*}
			T_\mu(\sfv^\pm,\sfq^\pm)\tnu = -\sfq^\pm \tnu +\mu
			\begin{pmatrix}
				\partial_\tnu \sfv_1^\pm +	\partial_{\tilde\tau}\sfv_2^\pm  \\
				\partial_\tnu \sfv_2^\pm - 	\partial_{\tilde\tau}\sfv_1^\pm 
			\end{pmatrix} ,
		\end{equation*}		 
and, since $[\partial_{\tilde\tau}\sfv]=0$ as a consequence of \eqref{eq:regcla} and~\eqref{eq:Stokes_zero}$_3$, we arrive together with \eqref{eq:Stokes_zero}$_4$ at 	 
		\begin{equation}\label{eq:tensor_jump}
			\mu[\partial_{\tnu}\sfv]-[\sfq]\tnu=[T_\mu(\sfv,\sfq)]\tnu = 0.
		\end{equation}
	Set $(\sfv,\sfq):= \mathbf{1}_{\Omega^{+}}(\sfv^+ ,\sfq^+ )+\mathbf{1}_{\Omega^{-}}(\sfv^- ,\sfq^- )\in\Lp{\infty}{\Ss\times\RR,\RR^2\times\RR}.$
		We then compute, in light of~\eqref{eq:regcla}, \eqref{eq:Stokes_zero}$_1$-\eqref{eq:Stokes_zero}$_3$, and~\eqref{eq:tensor_jump},  that 
			\begin{equation*}
	\left.
	\begin{array}{rclll}
			\mu \Delta \sfv  -\nabla \sfq&=&0 ,\\[1ex]
	\vdiv \sfv&=& 0
	\end{array}	
	\right\}\qquad		\mbox{in $\mcD'(\Ss\times\RR)$}.
	\end{equation*}	
	In particular, taking the divergence of the first equation yields $\Delta \sfq=0$,  hence,  $\sfq$ is a harmonic function in $\Ss\times\RR$. 
	Since $\sfq$ is bounded by \eqref{eq:Stokes_zero}$_5$, Liouville's theorem and \eqref{eq:Stokes_zero}$_5$ now yield~${\sfq=0}$ in~$\RR^2$. 
	This in turn means that $\sfv_1$ and $\sfv_2$ are harmonic in $\Ss\times\RR$, and, since~$\sfv$ is bounded by \eqref{eq:Stokes_zero}$_5$, we conclude together with \eqref{eq:Stokes_zero}$_5$ that  $\sfv=0$ and ${\sf c_1}={\sf c_2}=0$, which proves  the uniqueness claim.\medskip

	\noindent{\bf Solution of the Stokes equations.} To prove that $(v^\pm,q^\pm)$ defined in~\eqref{eq:vq}-\eqref{eq:vqdef} solves the equations \eqref{eq:refStokes}$_1$-\eqref{eq:refStokes}$_2$, we fix $x_0\in\Omega^\pm$ and choose $\ve>0$ such that the closed ball
	$\overline{B}_\ve(x_0) $ is contained in $\Omega^\pm$.
	Since $(\mcU,\mcP)(\cdot-(s,f(s)))\in\rmC^\infty(\Omega^\pm,\RR^{2\times 2}\times\RR^2)$, cf. \eqref{eq:GPi} and \eqref{eq:UP}, for each fixed~${s\in\Ss}$, the partial derivatives			
	$\partial^{\alpha}_{x}\mcU^k_j(\cdot-(s,f(s)))$, $\partial^{\alpha}_{x}\mcP^k(\cdot-(s,f(s)))$, $\alpha\in\NN^2$,  are bounded  
	 in $\overline{B}_\ve(x_0)$  uniformly in $s\in\Ss.$ 
	  Therefore, the function $(v^\pm,q^\pm)$ is well-defined in~\eqref{eq:vq}-\eqref{eq:vqdef} and 
we may interchange  differentiation with respect to $x$ and the integral sign in these formulas.
Recalling that  $(\mcU^k,\mcP^k)$, $k=1,\, 2$,  solve the Stokes equations \eqref{eq:HSE} pointwise in $(\Ss\times\RR )\setminus\{0\}$, it follows immediately that $(v^\pm,q^\pm)$ solve \eqref{eq:refStokes}$_1$-\eqref{eq:refStokes}$_2$ 
in $\Omega^\pm$. \medskip

	\noindent{\bf Boundary conditions.} The boundary conditions \eqref{eq:refStokes}$_3$-\eqref{eq:refStokes}$_4$ 
	together with the  far-field boundary condition~\eqref{eq:refStokes}$_5$ for~${(v^\pm,q^\pm)}$  follow by combining the results of Lemma~\ref{Lem:v_bd}  and  Lemma~\ref{Lem:q_bd} below.
	\end{proof}

\section{The evolution problem}\label{Sec:3}

 In this section we combine the results  from Section~\ref{Sec:FTP}, Appendix~\ref{Sec:A}, and Appendix~\ref{Sec:B} to reformulate the moving boundary problem \eqref{eq:STOKES} 
as a fully  nonlinear and nonlocal evolution problem for the function $f$, see \eqref{eq:ev_eq} below, with nonlinearities   defined in terms of (singular) integral operators.
Then,   exploiting estimates from Appendix~\ref{Sec:A} and the localization result in Lemma~\ref{Lem:Cnm_approx_a}, we prove that the evolution problem~\eqref{eq:ev_eq} 
is of parabolic type. 
This property, together with the abstract parabolic theory from~\cite{Lunardi.1995}, is then used to establish our main results in Theorem~\ref{Thm:main} and Theorem~\ref{Thm:stability}.

\subsection{Two classes of (singular) integral operators}\label{Sec:31} 
In this section we introduce  two classes of (singular) integral operators $B_{n,m}^{p,q}$ and $C_{n,m}$, 
the operators $B_{n,m}^{p,q}$ (together with the integral operator $B_0$) constituting via \eqref{eq:B_by_Bnmpq} the main building blocks of the  evolution operator in~\eqref{eq:ev_eq}, 
while the operators $C_{n,m}$  (which in a suitable sense retain the singular part  of the operators $B_{n,m}^{p,q}$ with~$p=0$) are important in the analysis of \eqref{eq:ev_eq}.

To start, given integers $m,\,n,\,p,\,q\in\NN_0$ satisfying $p\leq n+q+1$, and Lipschitz continuous mappings
	  $\bfa=(a_1,\dots,a_m):\RR\to\RR^m,$ $\bfb=(b_1,\dots,b_n):\RR\to\RR^n,$~${\bfc=(c_1,\dots,c_q):\RR\to\RR^q}$ we define the integral operators
	\begin{equation}\label{eq:Bnmpq}
		B_{n,m}^{p,q}(\bfa\vert \bfb)[\bfc,\varphi](\xi):= \frac{1}{2\pi}\PV\int_{-\pi}^{\pi}\frac{\prod\limits_{i=1}^{n}\frac{\T{\xi,s}b_i}{\tss}\prod\limits_{i=1}^{q}\frac{\dg{\xi,s}{c_i}/2}{\tss}}{\prod\limits_{i=1}^{m}\sqp{1+\p{\frac{\T{\xi,s}a_i}{\tss}}^2}}\frac{\varphi(\xi-s)}{\tss}\tss^p\dx{s}
	\end{equation}
	and
	\begin{equation}\label{eq:Cnm}
		C_{n,m}(\bfa)[\bfb,\varphi](\xi):= \frac{1}{\pi}\PV\int_{-\pi}^{\pi}\frac{\prod\limits_{i=1}^{n}\frac{\dg{\xi,s}{b_i}}{s}}{\prod\limits_{i=1}^{m}\sqp{1+\p{\frac{\dg{\xi,s}{a_i}}{s}}^2}}\frac{\varphi(\xi-s)}{s}\dx{s}, 
	\end{equation}
	where $\varphi\in\Lp{2}{\Ss}$ and $\xi\in\RR$.
	We use the notation introduced in \eqref{eq:notation1} together with the shorthand 
	\begin{align}\label{eq:notation2}
			&\dg{\xi,s}{f}:= f(\xi)-f(\xi-s)\qquad\text{and}\qquad\T{\xi,s}f:=\tanh\Big(\frac{\dg{\xi,s}{f}}{2}\Big),\quad \xi,\, s\in\RR.
	\end{align}	
	As shown in Lemma~\ref{Lem:Anmq_Bnmpq_inf} below, the $\PV$ symbol is not needed in \eqref{eq:Bnmpq} if $p\geq1$.

Moreover, we point out that if  the functions $\bfa,\, \bfb,$ and $\bfc$ are $2\pi$-periodic, then so are also the mappings $B_{n,m}^{p,q}(\bfa\vert \bfb)[\bfc,\varphi]$
and $C_{n,m}(\bfa)[\bfb,\varphi]$.
In particular,
\begin{equation}\label{eq:HT}
 B_{0,0}^{0,0}[\varphi](\xi)=\frac{1}{2\pi}\PV\int_{-\pi}^\pi\frac{\varphi(\xi-s)}{t_{[s]}}\, \dx{s}=   H[\varphi](\xi),\qquad \xi\in\RR,
\end{equation}
 where $ H$ is the periodic Hilbert transform, see e.g. \cite{Torchinsky.2004, Butzer.1971}.
Mapping properties for the   operators $B_{n,m}^{p,q}$ and $C_{n,m}$ are established  in Appendix~\ref{Sec:A}. 

If all coordinate functions of $\bfa,$ $\bfb,$ and~$\bfc$ are  identical to a given function $f\in \Wpp{1}{\infty}{\Ss}$, we set
 \begin{equation}\label{eq:B(f)}
  B_{n,m}^{p,q}(f):=B_{n,m}^{p,q}(f,\ldots, f\vert f,\ldots,f)[f,\ldots,f,\cdot],\qquad 0\leq p\leq n+q+1,
 \end{equation}
respectively,
  \begin{equation}\label{eq:Cnm0}
	C_{n,m}^0(f):=C_{n,m}(f,\ldots,f)[f,\ldots, f,\cdot].
	\end{equation}
 The operators $ B_{n,m}^{p,q}(f)$ appear in the reformulation~\eqref{eq:ev_eq} of the Stokes problem and  $C_{n,m}^0(f)$ are used in its analysis.

Finally, we introduce a further integral operator $B_0$ by setting
	\begin{equation}\label{eq:B0_alt}
		B_0(f)[\varphi](\xi):=\frac{1}{2\pi}\int_{-\pi}^\pi\ln\!\p{\frac{\tss^2+(\Txsf)^2}{(1+\tss^2)(1-(\Txsf)^2)}}\varphi(\xi-s)\dx{s},\qquad \xi\in\Ss,
	\end{equation}
	where again $f\in \Wpp{1}{\infty}{\Ss}$ and $\varphi\in\Lp{2}{\Ss}$.

As shown in Corollary~\ref{C:CCC},  given $r\in(3/2,2)$,  the mappings
\begin{equation}\label{eq:B_map_diff}
\begin{aligned}
&[f\mapsto B_{n,m}^{0,q}(f)]: \Hp{r}{\Ss}\to \mcL(\Hp{r-1}{\Ss}),\\[1ex]
&[f\mapsto B_0(f)],\,[f\mapsto B_{n,m}^{p,q}(f)]: \Hp{r}{\Ss}\to \mcL(\Hp{r-1}{\Ss}, \Hp{r}{\Ss}),\qquad   1\leq p\leq n+q+1,\\[1ex]
\end{aligned}
\end{equation}
are smooth.
 These properties are essential in the study of \eqref{eq:ev_eq}. 
 
 Finally, we introduce  the operators 
	\begin{equation}\label{eq:B_by_Bnmpq}
	\begin{aligned}
		B_1(f)&:=B_{0,1}^{0,0}(f)-B_{2,1}^{2,0}(f),\\
		B_2(f)&:=B_{1,1}^{0,0}(f)+B_{1,1}^{2,0}(f),\\
		B_3(f)&:=B_{0,2}^{0,1}(f)+B_{0,2}^{2,1}(f)-B_{2,2}^{0,1}(f)-2B_{2,2}^{2,1}(f)-B_{2,2}^{4,1}(f)+B_{4,2}^{2,1}(f)+B_{4,2}^{4,1}(f),\\	
		B_4(f)&:=B_{1,2}^{0,1}(f)+B_{1,2}^{2,1}(f)-B_{3,2}^{2,1}(f)-B_{3,2}^{4,1}(f),\\
		B_5(f)&:=2\big(B_{0,1}^{1,1}(f)-B_{2,1}^{3,1}(f)\big),\\
		B_6(f)&:=2\big(B_{1,1}^{1,1}(f)+B_{1,1}^{3,1}(f)\big),
	\end{aligned}
	\end{equation}
	which appear in a natural way in the analysis, see \eqref{eq:v_g} and \eqref{eq:v_z}.

\subsection{The reformulation of the Stokes problem~\eqref{eq:STOKES}}\label{Sec:32}
Let ~${(f,v^\pm,q^\pm)}$  be a solution  to~\eqref{eq:STOKES} enjoying the regularity properties  in Theorem~\ref{Thm:main}~(i).
Since $f(t)\in \Hp{3}{\Ss}$, and consequently  $ G(f(t))\in \Hp{1}{\Ss}^2$, see \eqref{defHG}, for all $t>0$, we infer from Remark~\ref{R:FTP} and Theorem~\ref{Thm:FT_unique}   that 
   the function 
$(v^\pm(t),q^\pm(t))$ is  identified by the system~\eqref{eq:Stokes}$_1$-\eqref{eq:Stokes}$_5$ and~\eqref{eq:c_inf} according to \eqref{eq:vq}-\eqref{eq:vqdef}.
Together with the kinematic boundary condition~\eqref{eq:Stokes}$_6$ and the  formulas~\eqref{eq:v_g}, and \eqref{eq:v_z} for the trace of the velocity $v_G$, we deduce that $f$ solves the evolution problem
\begin{equation}\label{eq:ev_eq}
		\frac{\dx{f}}{\dx{t}}(t)=\Psi(f(t)),\quad t>0,\qquad f(0)=f_0,
	\end{equation}
where the operator $\Psi$ is defined by
	\begin{equation}\label{eq:def_Psi}
		\Psi(f):= \frac{\sigma}{4\mu}f' \Psi_1(f)+\frac{\Theta}{4\mu}f' \Psi_3(f)-\frac{\sigma}{4\mu}\Psi_2(f)+\frac{\Theta}{4\mu}\Psi_4(f)+\frac{\Theta\ln(4)}{4\mu}\langle f\rangle, 
	\end{equation}
	with
	\begin{equation}\label{eq:def_Psii}
	\begin{aligned}
		\Psi_1(f)&:=\big(B_1-2B_4\big)(f)[\phi_1(f)-f' \phi_2(f)]\\
		&\quad\,+\big(2B_2+B_3)(f)[f' \phi_1(f)]+B_3(f)[\phi_2(f)],\\
	\Psi_2(f)&:= B_1(f)[\phi_2(f)-f' \phi_1(f)]\\
	&\quad\,+B_3(f)[\phi_1(f)-f' \phi_2(f)]+2B_4(f)[f' \phi_1(f)+\phi_2(f)],\\
	\Psi_3(f)&:=\big(B_0(f)+B_6(f)\big)[ff']+B_5(f)[f],\\
	\Psi_4(f)&:=\big(B_0(f)-B_6(f)\big)[f]+B_5(f)[ff'],
	\end{aligned}
	\end{equation}
	where $\phi=\phi(f)$ is given by
	\begin{equation}\label{eq:defphi}
		 \phi(f):=(\phi_1(f),\phi_2(f)):=(\omega(f)^{-1}- 1,f'\omega(f)^{-1}).
	\end{equation}
 The function $\phi=\phi(f)$ appears in the definition of $G=G(f) $ in \eqref{defHG}.
	
	Let  $r\in(3/2,2)$ be fixed in the following.
	As an important observation we note  that the right-hand side of \eqref{eq:ev_eq} is well-defined for all functions $f$ which belong to $ \Hp{r}{\Ss}$.
	In order to study \eqref{eq:ev_eq}, we first establish the following result.
	\begin{Lemma}\label{Lem:deriv_phi}
		Given $r\in (3/2,2)$, we have $\phi \in\rmC^\infty(\Hp{r}{\Ss},\Hp{r-1}{\Ss}´^2)$ and the Fr\' echet derivative~${\partial\phi(f_0)=(\partial\phi_1(f_0),\partial\phi_2(f_0))}$,~${f_0\in\Hp{r}{\Ss}}$, 
		satisfies
		\begin{equation}\label{eq:deriv_phi}
		\partial\phi_i(f_0)=a_i(f_0)\dfx{}{\xi}\in\mcL(\Hp{r}{\Ss},\Hp{r-1}{\Ss}),\qquad i=1,\, 2,
		\end{equation}
		 where $a_i(f_0)\in \Hp{r-1}{\Ss}$ are given by
		\begin{equation}\label{eq:def_ai}
			a_1(f_0):=-\frac{f_0'}{(1+f_0'^2)^{3/2}}\qquad\text{and}\qquad a_2(f_0):=\frac{1}{(1+f_0'^2)^{3/2}}.
		\end{equation}
	\end{Lemma}	
		\begin{proof}
		 The proof is  similar to that of \cite[Lemma~3.5]{Matioc.2021}.
	\end{proof}

	Combining \eqref{eq:B_by_Bnmpq}, \eqref{eq:def_Psi}, \eqref{eq:def_Psii}, Lemma~\ref{Lem:deriv_phi}, and Corollary~\ref{C:CCC}, we conclude that
	  	\begin{equation}\label{eq:PSI_smooth}
			\Psi\in\rmC^\infty (\Hp{r}{\Ss},\Hp{r-1}{\Ss}).
		\end{equation}
		
\subsection{The Fr\'echet derivative}\label{Sec:33}	To apply the abstract parabolic theory from \cite[Chapter~8]{Lunardi.1995} in the context of~\eqref{eq:ev_eq}, which we now 
view as an evolution equation in the ambient space~${\Hp{r-1}{\Ss}}$,
	it    remains to show that the Fr\'echet derivative $\partial\Psi(f_0)\in\mcL(\Hp{r}{\Ss},\Hp{r-1}{\Ss})$ 
 generates a strongly continuous  analytic semigroup in $\mcL(\Hp{r-1}{\Ss})$.
 This is the content of the next result (where we use  notation from \cite{Amann.1995}).
 \begin{Proposition}\label{Prop:PSI_Gen}
		Given $f_0\in\Hp{r}{\Ss}$, we have
		\begin{equation}\label{eq:PSI_Gen}
			-\partial\Psi(f_0)\in\mcH(\Hp{r}{\Ss},\Hp{r-1}{\Ss}).
		\end{equation}
	\end{Proposition}
	
	For the remainder of this section and in Section~\ref{Sec:34}, we fix $f_0\in \Hp{r}{\Ss}$, $r\in(3/2,2)$. 
	The proof of Proposition \ref{Prop:PSI_Gen} requires some preparation. To start, we infer from \eqref{eq:def_Psi} that 
\begin{equation} \label{der:Psi}	
	\begin{aligned}
			\partial\Psi(f_0)[f]&= \frac{1}{4\mu}\big(\sigma \Psi_1(f_0)+ \Theta\Psi_3(f_0)\big)f'+\frac{\sigma}{4\mu}\big(f_0' \partial\Psi_1(f_0)-\partial\Psi_2(f_0)\big)[f]\\[1ex]
			&\quad+\frac{\Theta}{4\mu}\big(\partial\Psi_4(f_0)+f_0' \partial\Psi_3(f_0)\big)[f]+\frac{\Theta\ln(4)}{4\mu}\langle f\rangle, \qquad f\in \Hp{r}{\Ss}.
	\end{aligned}
	\end{equation}
	
	The terms on the second line of \eqref{der:Psi} are lower order perturbations.
	To quantify this, let~${r'\in(3/2,r)}$ be fixed in the following. 
	Recalling \eqref{eq:B_by_Bnmpq} and \eqref{eq:def_Psii},  Corollary~\ref{C:CCC} (with $r$ replaced by $r'$) yields
	\begin{equation*}
	\partial B_i(f_0)\in\mcL(\Hp{r'}{\Ss},\mcL(\Hp{r'-1}{\Ss},\Hp{r'}{\Ss})),\qquad i\in\{0,\, 5,\, 6\},
	\end{equation*}
	and therefore
\begin{equation}\label{eq:DPsi_34}
		\norm{\partial\Psi_i(f_0)[f]}_{\nHp{r-1}}\leq C\norm{f}_{\nHp{r'}},\qquad f\in\Hp{r}{\Ss},\quad i\in\{3,\, 4\}.
	\end{equation}
	Moreover, we clearly have
	\begin{equation}\label{eq:DPsi_lt}
		\norm{\langle f\rangle}_{\nHp{r-1}}\leq C\|f\|_1\leq C\norm{f}_{\nHp{r'}},\qquad f\in\Hp{r}{\Ss}.
	\end{equation}
	
	The first two terms on the right-hand side of \eqref{der:Psi} are easy to handle as they are first order differential operators, however the next two terms are more intricate.
	To analyze them we first compute, for some fixed $\varphi_0\in \Hp{r-1}{\Ss}$, by combining \eqref{eq:B_by_Bnmpq}, \eqref{eq:B=A+C}, \eqref{eq:Cnm_com_Hr-1_Hr-1}, \eqref{eq:Frechet_Bnmpq},    
	 Lemma~\ref{Lem:Anmq_Bnmpq_inf}, and Corollary~\ref{C:CCC}  that
	\begin{equation}\label{derbs}
	\begin{aligned}
		\partial B_1(f_0)[f][\varphi_0]&=-2\varphi_0C^0_{1,2}(f_0)[f']+R_1[f],\\
		\partial B_2(f_0)[f][\varphi_0]&= \varphi_0\big(C^0_{0,1}-2C_{2,2}^0\big)(f_0)[f']+R_2[f],\\
			\partial B_3(f_0)[f][\varphi_0]&= \varphi_0\big(C^0_{0,2}-3C_{2,2}^0-4C^0_{2,3}+4C_{4,3}^0\big)(f_0)[f']+R_3[f],\\
			\partial B_4(f_0)[f][\varphi_0]&= \varphi_0\big(2C^0_{1,2}-4C_{3,3}^0\big)(f_0)[f']+R_4[f] \\
	\end{aligned}
\end{equation}	 
for all $f\in \Hp{r}{\Ss}$, where 
	\begin{equation}\label{derbsrest}
		\|R_i[f]\|_{\nHp{r-1}}\leq   C\norm{f}_{\nHp{r'}},\qquad f\in\Hp{r}{\Ss},\quad 1\leq i\leq 4.
	\end{equation}
	Moreover,  \eqref{eq:B_by_Bnmpq}, \eqref{eq:B=A+C},   Lemma~\ref{Lem:Anmq_Bnmpq_inf}, and Corollary~\ref{C:CCC} entail that 
		\begin{equation}\label{aprbs}
	\begin{aligned}
		 B_1(f_0)[\varphi]&=C^0_{0,1}(f_0)[\varphi]+\tilde R_1[\varphi],\\
		 B_2(f_0)[\varphi]&=C^0_{1,1}(f_0)[\varphi] +\tilde R_2[\varphi],\\
		 B_3(f_0)[\varphi]&=  \big(C^0_{1,2}-C_{3,2}^0\big)(f_0)[\varphi]+\tilde R_3[\varphi],\\
		 B_4(f_0)[\varphi]&=  C_{2,2}^0(f_0)[\varphi]+ \tilde R_4[\varphi]\\
	\end{aligned}
\end{equation}	 
for all $\varphi\in \Hp{r-1}{\Ss}$, where 
	\begin{equation}\label{aprbsrest}
		\|\tilde R_i[\varphi]\|_{\nHp{r-1}}\leq   C\norm{\varphi}_{\nHp{r'-1}},\qquad \varphi\in\Hp{r-1}{\Ss},\quad 1\leq i\leq 4.
	\end{equation}
	Setting 
	\begin{equation}\label{eq:defa_iphi_i}
	a_i:=a_i(f_0)\qquad\text{and}\qquad \phi_i:=\phi_i(f_0),\quad i=1,\, 2,
	\end{equation}
	see \eqref{eq:defphi} and \eqref{eq:def_ai}, 
	we infer from \eqref{eq:def_Psii}, \eqref{eq:deriv_phi}, \eqref{derbs}-\eqref{eq:defa_iphi_i}, and  
the algebraic relation 
\[
C_{n,m}^{0}(f_0)+C_{n+2,m}^{0}(f_0)=C_{n,m-1}^{0}(f_0),\qquad m\geq 1,
\]  that
	\begin{equation}\label{eq:DPsi}
		\partial\Psi_i (f_0)[f]=T_i(f_0)[f]+T_{i,{\rm lot}}(f_0)[f],\qquad i=1,\,2,
	\end{equation}
	where
	\begin{equation}\label{eq:T_ij}
	\begin{aligned}
	T_1(f_0)[f]&:= (C_{0,2}^{0}-C_{2,2}^{0})(f_0)[(a_1-\phi_2-f_0' a_2)f']+C_{1,2}^{0}(f_0)[(3(\phi_1+f_0' a_1)+a_2)f']\\
	&\quad\,+C_{3,2}^{0}(f_0)[(\phi_1+f_0' a_1-a_2)f']\\
	 &\quad\,+ \phi_1\p{3f_0' C_{0,3}^{0}-6C_{1,3}^{0}-6 f_0' C_{2,3}^{0}+2C_{3,3}^{0}-f_0' C_{4,3}^{0}}(f_0)[f']\\
	 &\quad\,+\phi_2\p{C_{0,3}^{0}+6f_0' C_{1,3}^{0}-6C_{2,3}^{0}-2f_0' C_{3,3}^{0}+C_{4,3}^{0}}(f_0)[f'],\\
		T_2(f_0)[f]&:= (C_{1,2}^{0}-C_{3,2}^{0})(f_0)[(a_1+\phi_2-f_0' a_2)f'] -C_{0,2}^{0}(f_0)[(\phi_1+f_0' a_1-a_2)f']\\
		&\quad\,+C_{2,2}^{0}(f_0)[(\phi_1+f_0' a_1+3a_2)f']\\
	 &\quad\,+ \phi_1\p{C_{0,3}^{0}+6f_0' C_{1,3}^{0}-6 C_{2,3}^{0}-2f_0' C_{3,3}^{0}+C_{4,3}^{0}}(f_0)[f']\\
		&\quad\,+\phi_2\p{-f_0' C_{0,3}^{0}+2 C_{1,3}^{0}+6f_0' C_{2,3}^{0} -6 C_{3,3}^{0}-f_0' C_{4,3}^{0}}(f_0)[f']
	\end{aligned}
	\end{equation}
	and
\begin{equation}\label{eq:Ti_lot}
		\|T_{i,{\rm lot}}(f_0)[f]\|_{\nHp{r-1}}\leq C\norm{f}_{\nHp{r'}}, \qquad f\in\Hp{r}{\Ss},\quad i=1,\,2.
	\end{equation}

	\subsection{Localization of the Fr\'echet derivative}\label{Sec:34}
	Based on the formulas for~$\partial\Psi(f_0)$ provided in  Section~\ref{Sec:33} and inspired by the papers \cite{E94, ES95, Matioc.2019}, we prove in this section that the Fr\'echet derivative 
	 $\partial\Psi(f_0)$ can be locally approximated by certain Fourier multipliers which are themselves generators of strongly continuous analytic semigroups, see Proposition~\ref{Prop:loc} and~\eqref{eq:l-A_iso}-\eqref{eq:l-A}. 
The proof of Proposition~\ref{Prop:PSI_Gen} relies heavily on these properties  and concludes this section.

	To start, we choose for each $\ve\in(0,1)$ a set of functions $\{\pi_{j}^{\ve}: 1\leq j\leq N\}\subset \rmC^\infty(\Ss,[0,1])$, where the integer~$N=N(\ve)$ is sufficiently large,  such that
	\begin{equation}\label{eq:pi_jp}
	\begin{aligned}
		\bullet & \,\supp \pi_j^\ve= I_j^\ve+2\pi \ZZ\text{ with } I_j^\ve:= [x_j^\ve-\ve,x_j^\ve+\ve] \text{ and } x_j^\ve:=j\ve;\\
		\bullet & \, \sum_{j=1}^{N}\pi_j^\ve =1 \text{ in } \rmC^\infty(\Ss). 
	\end{aligned}
	\end{equation}	 
	We call $\{\pi_{j}^{\ve}: 1\leq j\leq N\}$ a $\ve$\emph{-partition of unity}. 
	Moreover,  associated to each $\ve$-partition of unity, we choose a further set $\{\chi_j^\ve: 1\leq j \leq N\}\subset\rmC^\infty(\Ss,[0,1])$ with
	\begin{equation}\label{eq:chi_jp}
	\begin{aligned}
		\bullet & \,\supp \chi_j^\ve =J_j^\ve+2\pi \ZZ \text{ with } J_j^\ve=[x_j^\ve-2\ve,x_j^\ve+2\ve] ;\\
		\bullet & \,\chi_j^\ve =1 \text{ on } \supp \pi_j^\ve.
	\end{aligned}
	\end{equation}
	Each $\ve$-partition of unity allows us to define a new norm    on $\Hp{s}{\Ss}$, $s\geq 0$, via the mapping
	 \[	 
\bigg[f\mapsto \sum_{j=1}^{N}\norm{\pi_j^\ve f}_{\nHp{s}}\bigg]	:  \Hp{s}{\Ss}\to\RR,
	 \]
	which is equivalent to the standard norm.
	Indeed, it is straightforward to show there exists a constant $c=c(\ve,s)\in(0,1)$ such that
	\begin{equation}\label{eq:norm_pi_jp}
		c\norm{f}_{\nHp{s}}\leq \sum_{j=1}^{N}\norm{\pi_j^\ve f}_{\nHp{s}}\leq c^{-1}\norm{f}_{\nHp{s}},\qquad f\in  \Hp{s}{\Ss}.
	\end{equation}

	Following \cite{Matioc.2021}, we introduce the continuous path
	$
		\Phi: [0,1]\to\mcL(\Hp{r}{\Ss},\Hp{r-1}{\Ss})
	$
	 defined by
	\begin{equation}\label{eq:Phi_tau}
	\begin{aligned}
		\Phi(\tau)&:=\frac{\tau}{4\mu}\big(\sigma \Psi_1(f_0)+ \Theta\Psi_3(f_0)\big)\dfx{}{\xi}+\frac{\sigma}{4\mu}\big(\tau f_0' \partial\Psi_1(\tau f_0)-\partial\Psi_2(\tau f_0)\big) \\[1ex]
			&\,\quad+\frac{\tau\Theta}{4\mu}\big(\partial\Psi_4(f_0)+f_0' \partial\Psi_3(f_0)\big) +\frac{\tau\Theta\ln(4)}{4\mu}\langle\, \cdot\,\rangle,\qquad \tau\in [0,1],	
	\end{aligned}
	\end{equation} 
which connects the Fr\'echet derivative~$\partial\Psi(f_0)=\Phi(1)$, see \eqref{der:Psi}, to the Fourier multiplier
\begin{equation}\label{eq:Phi0}
		\Phi(0)= -\frac{\sigma}{4\mu} H\circ\dfx{}{\xi}=-\frac{\sigma}{4\mu}\p{-\frac{\dx{}^2}{\dx{\xi}^2}}^{1/2},
	\end{equation}
	see \eqref{eq:HT}. 
	In \eqref{eq:Phi0}, we use the  fact that the periodic Hilbert transform $H$ is a Fourier multiplier with
symbol $(-i\, {\rm sign}(k))_{k\in\mathbb{Z}}$.
 
	 The homotopy $\Phi$ will be used to conclude invertibility of $\lambda-\Phi(1)$ from $\lambda-\Phi(0)$ for $\Rel\lambda$ large enough.
	   In the arguments below we use the estimate
		\begin{equation}\label{eq:Hr_BAlg}
			\norm{fg}_{\nHp{s}}\leq C(\norm{f}_\infty \norm{g}_{\nHp{s}}+\norm{g}_\infty \norm{f}_{\nHp{s}}),\qquad \text{with $s\in(1/2,1)$},
		\end{equation}
		which holds for all $f,g\in\Hp{s}{\Ss}$. 
		The following proposition shows that the operator $\Phi(\tau)$ can be locally approximated by certain Fourier multipliers for all $\tau\in[0,1]$.
	\begin{Proposition}\label{Prop:loc}
		Let $\gamma>0$ and $ 3/2<r'<r<2$. 
		Then, there exists $\ve\in(0,1)$ together with  an~$\ve$-partition of unity~${\{\pi_{j}^{\ve}: 1\leq j\leq N\}}$,  a constant $K=K(\ve)$, and bounded operators
		\begin{equation*}
			\Aa_{j,\tau}\in\mcL(\Hp{r}{\Ss},\Hp{r-1}{\Ss}),\qquad j\in\set{1,\dots,N}, \quad\tau\in[0,1],
		\end{equation*}
		such that
		\begin{equation}\label{eq:Phi_AA_approx}
			\norm{\pi_j^\ve \Phi(\tau)[f]-\Aa_{j,\tau}[\pi_j^\ve f]}_{\nHp{r-1}}\leq \gamma\norm{\pi_j^\ve f}_{\nHp{r}}+K\norm{f}_{\nHp{r'}},
		\end{equation}
		for all $j\in\{1,\dots,N\}$, $f\in\Hp{r}{\Ss}$, and $\tau\in[0,1]$. The operators $\Aa_{j,\tau}$ are defined by
		\begin{equation*}
			\Aa_{j,\tau}:= -\alpha_{\tau}(x_j^\ve)\p{-\frac{\dx{}^2}{\dx{\xi}^2}}^{1/2}+\beta_\tau (x_j^\ve)\frac{\dx{}}{\dx{\xi}},\qquad j\in\{1,\dots,N\},\quad\tau\in[0,1],
		\end{equation*}
		with the functions $\alpha_\tau$, $\beta_\tau$ given by (see \eqref{nutauzomega})
		\begin{equation}\label{alphabeta}
			\alpha_{\tau}:= \frac{\sigma}{4\mu}\omega^{-1}(\tau f_0) \qquad\text{and}\qquad \beta_\tau :=\frac{\tau}{4\mu}\big(\sigma \Psi_1(f_0)+ \Theta\Psi_3(f_0)\big).
		\end{equation}
	\end{Proposition}
	\begin{proof}
	Let $\ve\in(0,1)$ and let~${\{\pi_{j}^{\ve}: 1\leq j\leq N\}}$ be an~$\ve$-partition of unity with the associated set~$\{\chi_j^\ve: 1\leq j \leq N\} $ satisfying \eqref{eq:chi_jp}.
In the following, we use the symbol $C$ for constants that are independent of $\ve$, and denote constants that depend on $\ve$ by $K$.
Recalling \eqref{eq:DPsi_34} and \eqref{eq:DPsi_lt}, the algebra property of $\Hp{r-1}{\Ss}$ leads us to
\begin{equation}\label{Step1}
		\norm{\pi_j^\ve\p{\frac{\tau \Theta}{4\mu}\big(\partial\Psi_4(f_0)+f_0' \partial\Psi_3(f_0)\big)[f]+\frac{\tau \Theta\ln(4)}{4\mu}\langle f\rangle}}_{\nHp{r-1}}\leq K\norm{f}_{\nHp{r'}} 
	\end{equation}
	for all $j\in\{1,\dots,N\}$, $f\in\Hp{r}{\Ss}$, and $\tau\in[0,1]$.

Moreover, since $\Psi_k(f_0)\in \Hp{r-1}{\Ss}\hookrightarrow \rmC^{r-3/2}(\Ss)$, $k\in\{1,\, 3\}$, we use \eqref{eq:chi_jp}$_2$ and \eqref{eq:Hr_BAlg}  to derive that  
\begin{align*}
\norm{\pi_j^\ve \beta_\tau f'-\beta_\tau(x_j^\ve)(\pi_j^\ve f)' }_{\nHp{r-1}}&\leq C\norm{(\Psi_1(f_0)-\Psi_1(f_0)(x_j^\ve))(\pi_j^\ve f)'}_{\nHp{r-1}}\\
		&\quad+C\norm{(\Psi_3(f_0)-\Psi_3(f_0)(x_j^\ve))(\pi_j^\ve f)'}_{\nHp{r-1}}+K\|f\|_{\nHp{r-1}},
\end{align*}
	where 
	\begin{align*}
		&C\norm{  (\Psi_k(f_0)-\Psi_k(f_0)(x_j^\ve))(\pi_j^\ve f)'}_{\nHp{r-1}}\\
& \leq C\norm{\chi_j^\ve (\Psi_k(f_0)-\Psi_k(f_0)(x_j^\ve))}_{\infty}\norm{(\pi_j^\ve f)'}_{\nHp{r-1}}+K\norm{f}_{\nHp{r'}}\\
		& \leq \frac{\gamma}{4} \norm{\pi_j^\ve f}_{\nHp{r}}+K\norm{f}_{\nHp{r'}},\qquad k\in\{1,\, 3\},
		\end{align*}
for all $j\in\{1,\dots,N\}$, $f\in\Hp{r}{\Ss}$, and $\tau\in[0,1]$, provided that $\ve $ is sufficiently small, and therefore
\begin{equation}\label{Step2}
 \norm{\pi_j^\ve \p{\frac{\tau\sigma}{4\mu}\Psi_1(f_0)+\frac{\tau\Theta}{4\mu}\Psi_3(f_0)}f'-\beta_\tau(x_j^\ve)(\pi_j^\ve f)' }_{\nHp{r-1}}\leq \frac{\gamma}{2} \norm{\pi_j^\ve f}_{\nHp{r}}+K\norm{f}_{\nHp{r'}}.
\end{equation}

Finally,  \eqref{eq:DPsi}-\eqref{eq:Ti_lot} and repeated use of Lemma~\ref{Lem:Cnm_approx_a} lead us to 
\begin{equation}\label{Step3}
 \Big\|\pi_j^\ve \frac{\sigma}{4\mu}\big(\tau f_0' \partial\Psi_1(\tau f_0)-\partial\Psi_2(\tau f_0)\big)[f] -\alpha_\tau(x_j^\ve) H[(\pi_j^\ve f)'] \Big\|_{\nHp{r-1}}
 \leq \frac{\gamma}{2} \norm{\pi_j^\ve f}_{\nHp{r}}+K\norm{f}_{\nHp{r'}}
\end{equation}
 for all $j\in\{1,\dots,N\}$, $f\in\Hp{r}{\Ss}$, and $\tau\in[0,1]$, provided that $\ve $ is sufficiently small.
 
 Gathering \eqref{Step1}-\eqref{Step3}, the claim follows in view of \eqref{eq:Phi_tau}.
	\end{proof}

Since $\Psi_k(f_0)\in\Hp{r-1}{\Ss}$, $k\in\{1,3\}$,  there exists a constant~$\eta\in (0,1)$ such that the functions~$\alpha_\tau$ and~$\beta_\tau$ defined in~\eqref{alphabeta} satisfy 
	\begin{equation*}
		\eta \leq \alpha_\tau \leq \eta^{-1}\qquad\text{and}\qquad |\beta_\tau|\leq \eta^{-1},\quad \tau\in [0,1].
	\end{equation*}
	Introducing  the Fourier multiplier
	\begin{equation*}
		\Aa_{\alpha,\beta}:= -\alpha \p{-\frac{\dx{}^2}{\dx{\xi}^2}}^{1/2}+\beta\frac{\dx{}}{\dx{\xi}}	\in\mcL(\Hp{r}{\Ss},\Hp{r-1}{\Ss}),\qquad \alpha\in [\eta,\eta^{-1}],\quad \beta\in[-\eta^{-1}, \eta^{-1}], 
	\end{equation*}
	it is straightforward to show, by using Fourier analysis techniques, that for all $ \alpha\in [\eta,\eta^{-1}]$ and  $\beta\in[-\eta^{-1}, \eta^{-1}]$,
	\begin{equation}\label{eq:l-A_iso}
		\text{$\lambda-\Aa_{\alpha,\beta}:\Hp{r}{\Ss}\to\Hp{r-1}{\Ss}$ is an  isomorphism for all $\Rel\lambda\geq 1$.}
	\end{equation}
	Moreover, there exists a constant $\kappa_0 =\kappa_0(\eta)\geq 1$ with the property that for all $ \alpha\in [\eta,\eta^{-1}]$ and~${\beta\in[-\eta^{-1}, \eta^{-1}]}$,
	\begin{equation}\label{eq:l-A}
		\kappa_0\norm{(\lambda-\Aa_{\alpha,\beta})[f]}_{\nHp{r-1}}\geq \abs{\lambda}\,\norm{f}_{\nHp{r-1}}+\norm{f}_{\nHp{r}}, \qquad f\in\Hp{r}{\Ss},\quad \Rel\lambda\geq 1.
	\end{equation}
	The relations \eqref{eq:l-A_iso}-\eqref{eq:l-A} imply in particular that the operator $\Aa_{\alpha,\beta}$ generates a strongly continuous  analytic semigroup, cf. \cite[Section~I.1.2]{Amann.1995}.
	Moreover, together with Proposition~\ref{Prop:loc} and the interpolation  property
	\begin{equation}\label{eq:interpolation}
			[\Hp{r_0}{\Ss},\Hp{r_1}{\Ss}]_{\theta}=\Hp{(1-\theta)r_0+\theta r_1}{\Ss},\qquad \theta\in (0,1),\quad -\infty<r_0 \leq r_1<\infty,  
		\end{equation}
	 where~$[\cdot,\cdot]_\theta$ is the complex interpolation functor, they enable us to prove Proposition~\ref{Prop:PSI_Gen}.
	 
	\begin{proof}[Proof of Proposition~\ref{Prop:PSI_Gen}]
		Let $r' \in (3/2,r)$ and $\kappa_0\geq 1$ be the constant in \eqref{eq:l-A}.
		 We may use Proposition~\ref{Prop:loc} with $\gamma:= 1/2\kappa_0$ to find $\ve\in(0,1)$, an $\ve$-partition of unity~${\{\pi_{j}^{\ve}: 1\leq j\leq N\},}$ a constant~$K=K(\ve)>0$, and operators~$\Aa_{j,\tau}\in\mcL(\Hp{r}{\Ss},\Hp{r-1}{\Ss})$, $1\leq j\leq N$ and $\tau\in [0,1]$, satisfying
		\begin{equation*}
			2\kappa_0 \norm{\pi_j^\ve \Phi(\tau)[f]-\Aa_{j,\tau}[\pi_j^\ve f]}_{\nHp{r-1}}\leq \norm{\pi_j^\ve f}_{\nHp{r}}+2\kappa_0 K\norm{f}_{\nHp{r'}},\qquad f\in\Hp{r}{\Ss}.
		\end{equation*}
		Furthermore, \eqref{eq:l-A} yields for all $1\leq j\leq N$, $\tau\in [0,1]$, and $ \Rel\lambda\geq 1$
		\begin{equation*}
			2\kappa_0\norm{(\lambda-\Aa_{j,\tau
			})[\pi_j^\ve f]}_{\nHp{r-1}}\geq 2\abs{\lambda}\norm{\pi_j^\ve f}_{\nHp{r-1}}+2\norm{\pi_j^\ve f}_{\nHp{r}}, \qquad f\in\Hp{r}{\Ss}.
		\end{equation*}
		Combining the above inequalities, we  get
		\begin{equation*}
		\begin{aligned}
			2\kappa_0 \norm{\pi_j^\ve (\lambda-\Phi(\tau))[f]}_{\nHp{r-1}}&\geq 2\kappa_0 \norm{(\lambda-\Aa_{j,\tau})[\pi_j^\ve f]}_{\nHp{r-1}}
			-2\kappa_0 \norm{\pi_j^\ve \Phi(\tau) [f]-\Aa_{j,\tau}[\pi_j^\ve f]}_{\nHp{r-1}}\\
			&\geq 2\abs{\lambda} \norm{\pi_j^\ve f}_{\nHp{r-1}}+\norm{\pi_j^\ve f}_{\nHp{r}}-2\kappa_0 K \norm{f}_{\nHp{r'}}.
		\end{aligned}
		\end{equation*}
		 In view of  \eqref{eq:norm_pi_jp}, \eqref{eq:interpolation},  and Young's inequality  we conclude that there exist constants~${\kappa\geq 1}$ and~$\omega>1$ such that
		\begin{equation}\label{eq:l-Phi}
			\kappa\norm{(\lambda-\Phi(\tau))[f]}_{\nHp{r-1}}\geq \abs{\lambda}\,\norm{f}_{\nHp{r-1}}+\norm{f}_{\nHp{r}}
		\end{equation}
		for all $\tau\in [0,1]$, $\Rel\lambda \geq \omega$, and $f\in\Hp{r}{\Ss}$.
		
		Since $\omega-\Phi(0)=\omega-\Aa_{\sigma/4\mu,0}$ is an isomorphism, see \eqref{eq:Phi0} and \eqref{eq:l-A_iso}, the method of continuity, cf.~\cite[Proposition~I.1.1.1]{Amann.1995}, 
		together with \eqref{eq:l-Phi} implies that $\omega-\Phi(1)=\omega-\partial \Psi(f_0)$ is also an isomorphism.
		This property combined with  \eqref{eq:l-Phi} (with $\tau=1$) leads us to the desired conclusion, see \cite[Section~I.1.2]{Amann.1995}. 
	\end{proof}	

	\subsection{Proof of the main results}\label{Sec:35}
This section is devoted to the proof of the main results in Theorem~\ref{Thm:main} and Theorem~\ref{Thm:stability}.

	\begin{proof}[Proof of Theorem~\ref{Thm:main}]
		The proof follows from \eqref{eq:PSI_smooth} and Proposition~\ref{Prop:PSI_Gen}, by using the abstract parabolic theory in \cite[Chapter~8]{Lunardi.1995}.
		Given the substantial resemblance of the arguments to those  in the non-periodic case,  see the proof  of \cite[Theorem 3.2]{Matioc.2021}, we  refrain from presenting them herein.
	\end{proof}
	
	It remains to establish Theorem~\ref{Thm:stability}.
For this, we define the Hilbert space
		\begin{equation*}
			\hHp{s}{\Ss}:=\set{f\in\Hp{s}{\Ss}:\langle f\rangle =0},\qquad s\geq 0,
		\end{equation*}
and infer from  \eqref{eq:pres}  and \eqref{eq:def_Psi} that $\Psi(f)\in \hHp{r-1}{\Ss}$ for all $f\in \Hp{r}{\Ss}$, $r\in(3/2,2)$.
Therefore, the mapping 
\[
\widehat\Psi:=\Psi\big|_{\hHp{r}{\Ss}}:\hHp{r}{\Ss}\to \hHp{r-1}{\Ss}
\]
  is well-defined and smooth, see \eqref{eq:PSI_smooth}.
  Moreover, for initial data  $f_0\in\hHp{r}{\Ss}$, the evolution problem \eqref{eq:ev_eq}  is equivalent to
	\begin{equation}\label{eq:ev_eq_mf}
		\frac{\dx{f}}{\dx{t}}(t)=\widehat \Psi(f(t)), \qquad t>0,\quad f(0)=f_0,
	\end{equation}
	which is also of parabolic type. 
	 Indeed, given $f_0\in\hHp{r}{\Ss}$, the   Fr\'{e}chet derivative $\partial\widehat\Psi(f_0) $ is the generator of a strongly continuous  analytic semigroup in $\mcL(\hHp{r-1}{\Ss})$.
This is a consequence of \cite[Corollary I.1.6.3]{Amann.1995}  since, observing that $\hHp{s}{\Ss} $ 	 is the orthogonal complement of the set of constant functions in~$\Hp{s}{\Ss} $, $s\geq 0$, 
we may interpret $\partial \Psi(f_0) $ as the matrix operator
\[
\partial \Psi(f_0) =\begin{bmatrix}
\partial\widehat\Psi(f_0) &0\\
0&0 
\end{bmatrix}: \hHp{r}{\Ss}\oplus\RR\to\hHp{r-1}{\Ss}\oplus\RR.
\]
	
It thus remains to study the stability properties of the zero solution to \eqref{eq:ev_eq_mf}.
This is advantageous, because in contrast to $\partial\Psi(0)$, the Fr\'echet derivative  $\partial\widehat\Psi(0)$ does not have zero as an eigenvalue, when assuming \eqref{eq:Stab}, as the next lemma shows.
	\begin{Lemma}\label{Lem:Psi_0_spec}
		The spectrum $\sigma(\partial\widehat\Psi(0))$ of   $\partial\widehat\Psi(0)\in\mcL(\hHp{r}{\Ss},\hHp{r-1}{\Ss})$ is given by
		\begin{equation}\label{spec00}
			\sigma(\partial\widehat\Psi(0))=\set{-\frac{\sigma k^2+\Theta}{4\mu|k|} :k\in\NN}.
		\end{equation}
	\end{Lemma}
	\begin{proof}
In view of~\eqref{eq:HT}, \eqref{eq:B0_alt}, \eqref{eq:B_by_Bnmpq}, \eqref{eq:def_Psii}, \eqref{eq:defphi},  \eqref{der:Psi},  \eqref{eq:Phi0}, and Lemma~\ref{Lem:deriv_phi} we have
\[
\partial\Psi(0)=\frac{\Theta}{4\mu}B_0(0)-\frac{\sigma}{4\mu}\p{-\frac{\dx{}^2}{\dx{\xi}^2}}^{1/2}. 
\]
The operator $B_0(0)$ is actually also a Fourier multiplier.
Indeed, letting~${S\in\mcL(\hHp{r-1}{\Ss})}$ denote the operator which associates to each  function $f\in \hHp{r-1}{\Ss}$ its antiderivative,
 that is
 \[
S[f](\xi):=\int_0^\xi f(s)\,\dx{s}+\frac{1}{2\pi}\int_0^{2\pi} sf(s)\,\dx{s},\qquad \xi\in\Ss,
 \]
integration by parts leads to
\[
B_0(0)[f](\xi)= \frac{1}{2\pi}\int_{-\pi}^\pi\ln \big(\sin^2(s/2)\big) f(\xi-s)\dx{s}= H[S[f]](\xi),\qquad \xi\in\Ss.
\]
The relation \eqref{spec00} is now an immediate consequence of the latter relation. 
	\end{proof}
	We are now in a position to prove Theorem~\ref{Thm:stability}, which is based  on asymptotic theory for abstract parabolic problems from \cite[Chapter~9]{Lunardi.1995}.
	\begin{proof}[Proof of Theorem~\ref{Thm:stability}]
		In order to establish (i), let  \eqref{eq:Stab} be satisfied. Then, in view of Lemma~\ref{Lem:Psi_0_spec}, we have
		\begin{equation*}
			\sup\set{\Rel\lambda:\lambda\in\sigma(\partial\widehat\Psi(0))}\leq -\vartheta_0<0.
		\end{equation*}
	Therefore, the assumptions of \cite[Theorem~9.1.2]{Lunardi.1995} are fulfilled in the context of the evolution problem~\eqref{eq:ev_eq_mf} and, together with Theorem~\ref{Thm:main}, we conclude Theorem~\ref{Thm:stability}~(i).
	
	Concerning (ii), assume  now that $ \sigma+\Theta<0$.
	Then
	\begin{equation*}
	\left\{
	\begin{array}{lll}
	-\cfrac{\sigma +\Theta}{4\mu}\in\sigma_+(\partial\widehat\Psi(0)):=\sigma(\partial\widehat\Psi(0))\cap \{\lambda\in\CC\,:\,\Rel\lambda>0\};\\[1ex]
	\inf\{\Rel\lambda\,:\, \lambda\in \sigma_+(\partial\widehat\Psi(0))\}>0.
	\end{array}
	\right.
	\end{equation*}
	A direct application of \cite[Theorem~9.1.3]{Lunardi.1995} provides the desired instability result.
	\end{proof}

\appendix
\section{Some classes  of (singular) integral operators}\label{Sec:A}
	 In this section we establish several important mapping properties  for the (singular) integral operators~$B_{n,m}^{p,q},\, C_{n,m}$,  
	  and~$B_0$ introduced in  \eqref{eq:Bnmpq}, \eqref{eq:Cnm}, and~\eqref{eq:B0_alt}.
	  
	 We start by relating   the two families $B_{n,m}^{p,q}$ (with $p=0$) and $ C_{n,m}$.
	  To this end  we define for integers~${m,\,n,\,p,\,q\in\NN_0}$ that satisfy~$p\leq n+q+1$, $\ell\in\set{1,2}$, and Lipschitz continuous mappings
	  $\bfa=(a_1,\dots,a_m):\RR\to\RR^m,$ $\bfb=(b_1,\dots,b_n):\RR\to\RR^n,$~${\bfc=(c_1,\dots,c_q):\RR\to\RR^q}$  the integral operator
	\begin{equation}\label{eq:Anmlq}
	\begin{aligned}
		&A_{n,m}^{\ell,q}(\bfa\vert \bfb)[\bfc,\varphi](\xi)\\
		&:=\frac{1}{2\pi}\int_{-\pi}^{\pi}\sqp{\frac{\prod\limits_{i=1}^{n}\frac{\T{\xi,s}b_i}{\tss}\prod\limits_{i=1}^{q}\frac{\dg{\xi,s}{c_i}/2}{\tss}}{\prod\limits_{i=1}^{m}
		\sqp{1+\p{\frac{\T{\xi,s}a_i}{\tss}}^2}}\frac{1}{\tss^{\ell}}-\frac{\prod\limits_{i=1}^{n}\frac{\dg{\xi,s}{b_i}/2}{s/2}\prod\limits_{i=1}^{q}\frac{\dg{\xi,s}{c_i}/2}{s/2}}{\prod\limits_{i=1}^{m}
		\sqp{1+\p{\frac{\dg{\xi,s}{a_i}/2}{s/2}}^2}}\frac{1}{(s/2)^\ell}}\varphi(\xi-s)\dx{s},
	\end{aligned}
	\end{equation}
	where $\varphi\in\Lp{2}{\Ss}$ and $\xi\in\RR$ (see \eqref{eq:notation1} and \eqref{eq:notation2}).
The following relation
	\begin{equation}\label{eq:B=A+C}
		B_{n,m}^{0,q}(\bfa\vert \bfb)[\bfc,\varphi]=A_{n,m}^{1,q}(\bfa\vert \bfb)[\bfc,\varphi]+C_{n+q,m}(\bfa)[(\bfb,\bfc),\varphi],\qquad m,\,n,\,q\in\NN_0,
	\end{equation}
	and the fact that the operators $A_{n,m}^{\ell,q}$ are regularizing, see Lemma~\ref{Lem:Anmq_Bnmpq_inf}, will enable us to transfer mapping properties obtained for the operators $C_{n,m}$, see Section~\ref{Sec:A1},
	to the operators~$B_{n,m}^{0,q}$ (which have kernels with a higher degree of nonlinearity than the former). 
	 We note that   $A_{n,m}^{\ell,q}(\bfa\vert \bfb)[\bfc,\varphi]$ is $2\pi$-periodic  if   $\bfa,\, \bfb,$ and $\bfc$ have this property.

	Before establishing mapping properties for these operators, we collect below some useful elementary inequalities 
	 \begin{equation}\label{eq:tanh}
	 	\abs{\tanh(x)}\leq \abs{x}\quad \text{and}\quad \abs{\tanh(x)-x}\leq \abs{x \tanh ^2(x)},
	 \end{equation}
	 \begin{equation}\label{eq:tan}
	 	\abs{y}\leq\abs{\tan(y)} \quad \text{and}\quad \abs{\tan(y)-y}\leq \abs{y^2 \tan(y)},
	 \end{equation}
	 \begin{equation}\label{eq:deriv}
		\abs{\frac{\Tf{x,s}{d}}{\tss}}\leq \abs{\frac{\dg{x,s}{d}/2}{\tss}}\leq \abs{\frac{\dg{x,s}{d}}{s}}\cdot\bigg|\frac {s/2}{\tss}\bigg|\leq \|d'\|_\infty\bigg|\frac {s/2}{\tss}\bigg|\leq \|d'\|_\infty,\qquad d\in \Wpp{1}{\infty}{\Ss},
	\end{equation}
	for $x\in\RR$, $y\in(-\pi/2,\pi/2)$ and $0\neq s \in(-\pi,\pi)$.
	
	When estimating these operators, the standard  norm on $\Hp{r}{\Ss}$, defined by means of  the Fourier transform, is not so practical. 
	Instead of using this norm, we recall the classical identity~${\Hp{s}{\Ss}=\Wp{s}{\Ss}}$ for all $s\geq 0$, cf. e.g. \cite{Schmeisser.1987}. 
	For non-integer  $s>0$ it holds that  
	\begin{equation*}
		\Wp{s}{\Ss}:=\set{v\in\Wp{[s]}{\Ss}:[v]_{\nWp{s}}<\infty},\qquad s=[s]+\set{s},\quad [s]\in\NN_0,~\set{s}\in(0,1),
	\end{equation*}
where
	\begin{equation*}
		[v]_{\nWp{s}}^2:= \int_{-\pi}^\pi\int_{-\pi}^\pi\frac{\abs{v^{([s])}(\xi+y)-v^{([s])}(\xi)}^{2}}{|y|^{1+2\set{s}}}\dx{\xi}\dx{y} 
		= \int_{-\pi}^\pi\frac{\norm{\tau_y v^{([s])}-v^{([s])}}^2_2}{|y|^{1+2\set{s}}}\dx{y} 
	\end{equation*}
	and 
	\begin{equation}\label{eq:tauy}
	\tau_y v:= v(\cdot+y)
	\end{equation} 
	is the left shift operator. 
	The space $\Wp{s}{\Ss}$ is equipped with the norm
	\begin{equation*}
		\norm{v}_{\nWp{s}}^2:= \norm{v}_{\nWp{[s]}}^2+[v]_{\nWp{s}}^2.
	\end{equation*}
	Using Plancherel's identity, it is easy to verify that the norms $\norm{\cdot}_{\nHp{s}}$ and~$\norm{\cdot}_{\nWp{s}}$ are equivalent.

	\subsection{Estimates for the  operators \texorpdfstring{$C_{n,m}$}{Cnm}}\label{Sec:A1}
	In Lemma~\ref{Lem:Cnm_est} we gather some useful  mapping properties  of the singular integral operators~$C_{n,m}$.
	
	\pagebreak
	
	\begin{Lemma}\label{Lem:Cnm_est} Let $n,\,m\in\NN_{0}$, $\bfa=(a_1,\dots,a_m):\RR\to\RR^m $, and $\bfb=(b_1,\dots,b_n):\RR\to\RR^n$.
		\begin{itemize}
\item[\rm{(i)}] Given $\bfa\in \Wpp{1}{\infty}{\RR}^m$, there exists a constant $ C>0$ that depends only on $n,\, m$, and~$\norm{\bfa'}_\infty$ such that for all $\bfb\in \Wpp{1}{\infty}{\RR}^n$ and $\theta\in \RR$
we have
	 	\begin{equation}\label{eq:Cnm_L2_L2}
	 		\norm{C_{n,m}(\bfa)[\bfb,\cdot]}_{\mcL(\Lp{2}{\Ss},\Lp{2}{(\theta-\pi,\theta+\pi)})}\leq C \prod_{i=1}^{n}\norm{b_i'}_{\infty}.
	 	\end{equation}
	 	Moreover, $C_{n,m}\in\rmC^{1-}(\Wpp{1}{\infty}{\Ss}^m,\mcL^{n}_{sym}(\Wpp{1}{\infty}{\Ss},\mcL(\Lp{2}{\Ss})))$.

	 	\item[\rm{(ii)}]  Given $n\in\NN,$  $r\in(3/2,2)$, $\tau\in(5/2-r,1),$ and $\bfa\in\Hp{r}{\Ss}^n$, there exists a constant~$ C>0$ that depends
	 	only on $n,\,m,\,r$, and~$ \norm{\bfa}_{\nHp{r}}$ (and on $\tau$ in \eqref{eq:Cnm_com_L2_Hr-1}), such that for all~${\bfb\in\Hp{r}{\Ss}^n}$ and $\varphi\in\Hp{r-1}{\Ss}$ we have
	 	\begin{equation}\label{eq:Cnm_L2_Hr-1}
	 		\norm{C_{n,m}(\bfa)[\bfb,\varphi]}_2\leq C \norm{b_1'}_2 \norm{\varphi}_{\nHp{r-1}}\prod_{i=2}^{n}\norm{b_i'}_{\nHp{r-1}}
	 	\end{equation}
	 	and
	 	\begin{equation}\label{eq:Cnm_com_L2_Hr-1}
	 	\begin{aligned}
	 		&\norm{C_{n,m}(\bfa)[\bfb,\varphi]-\varphi C_{n-1,m}(\bfa)[b_2,\dots,b_n,b_1']}_2\\
	 		&\quad\leq C\norm{b_1}_{\nHp{\tau}}\norm{\varphi}_{\nHp{r-1}}\prod_{i=2}^{n} \norm{b_i'}_{\nHp{r-1}}.
	 	\end{aligned}
	 	\end{equation}
	 	\item[\rm{(iii)}]  Given  $r\in(3/2,2)$ and $\bfa\in\Hp{r}{\Ss}^m$, there exists a constant $C>0$ that depends only on $n,\,m,\,r$, and $\norm{\bfa}_{\nHp{r}}$ 
	 	such that for all~${\bfb\in\Hp{r}{\Ss}^n}$ and $\varphi\in\Hp{r-1}{\Ss}$ we have
	 	\begin{equation}\label{eq:Cnm_Hr-1_Hr-1}
	 		\norm{C_{n,m}(\bfa)[\bfb,\varphi]}_{\nHp{r-1}}\leq C \norm{\varphi}_{\nHp{r-1}}\prod_{i=1}^{n}\norm{b_i'}_{\nHp{r-1}}.
	 	\end{equation}
	 	\item[\rm{(iv)}]  Given $n\in\NN$, $r\in(3/2,2)$, $r'\in(3/2,r)$, and $\bfa\in\Hp{r}{\Ss}^m$, there exists a constant~$C>0$ that depends only on $n,\,m,\,r,\,r',$ and $\norm{\bfa}_{\nHp{r}}$ such that
	 	for all~${\bfb\in\Hp{r}{\Ss}^n}$ and~${\varphi\in\Hp{r-1}{\Ss}}$ we have
	 	\begin{equation}\label{eq:Cnm_com_Hr-1_Hr-1}
	 	\begin{aligned}
	 		&\norm{C_{n,m}(\bfa)[\bfb,\varphi]-\varphi C_{n-1,m}(\bfa)[b_2,\dots,b_n,b_1']}_{\nHp{r-1}}\\
	 		&\quad\leq C\norm{b_1}_{\nHp{r'}}\norm{\varphi}_{\nHp{r-1}}\prod_{i=2}^{n} \norm{b_i}_{\nHp{r}}.
	 	\end{aligned}
	 	\end{equation}
	 	\end{itemize}
	 \end{Lemma} 
	 
	 \begin{proof}
	 	The claim (i) is established in \cite[Lemma~A.1]{Matioc.2020} in the case $\theta=0$.
	 	The result for~${\theta\neq0}$ is obtained from the result for~${\theta=0}$ via the identity
	 	\[C_{n,m}(\bfa)[\bfb,\varphi](\xi)=C_{n,m}(\tau_\theta\bfa)[\tau_\theta\bfb,\tau_\theta\varphi](\xi-\theta),\qquad \xi,\,\theta\in\RR.\]
	 	 Moreover, the proof of  (ii) is similar to that of \cite[Lemma~4]{Abels.2022} and we therefore omit it. 
	 	Finally,~(iii) and~(iv) can be established by arguing analogously as in the nonperiodic version of these results, cf. \cite[Lemma~5 and 6]{Abels.2022}. 
	\end{proof}
	
	\subsection{Estimates for the operators \texorpdfstring{$A_{n,m}^{\ell,q}$}{Anmlq} and \texorpdfstring{$B_{n,m}^{p,q}$}{Bnmpq}}\label{Sec:A2}
	We study the integral operators  $B_{n,m}^{p,q}$ and $A_{n,m}^{\ell,q}$ and show first in Lemma~\ref{Lem:Anmq_Bnmpq_inf} that 
	 $A_{n,m}^{\ell,q}$ regularizes, the same being true for~$B_{n,m}^{p,q}$   provided that~$p\geq1$, see Lemma~\ref{Lem:Bnmpq_H1_L2}.

	\begin{Lemma}\label{Lem:Anmq_Bnmpq_inf}
	Let $ n,\,m,\,p,\,\, q\in\NN_0$ with $1\leq p\leq n+q+1$, $\ell\in\set{1,2}$, $r\in(3/2,2)$, and
	 let~$(\bfa, \bfb,\bfc)\in\Wpp{1}{\infty}{\Ss}^{m+n+q}$ be given. Then
	\begin{equation}\label{eq:Anmq_Bnmpq_C}
A_{n,m}^{\ell,q}(\bfa\vert \bfb)[\bfc,\cdot],\ B_{n,m}^{p,q}(\bfa\vert \bfb)[\bfc,\cdot]\in\mcL(\Lp{1}{\Ss}, \rmC(\Ss)),
	\end{equation}
	and there exists  a constant $ C>0$ that depends only on $n,\,m,\, p,\, q$,   and $\norm{(\bfa',\bfb')}_\infty$ such that
	 for all $\varphi\in \Lp{1}{\Ss}$ we have
	\begin{equation}\label{eq:Anmq_Bnmpq_inf}
		\|A_{n,m}^{\ell,q}(\bfa\vert \bfb)[\bfc,\varphi]\|_\infty +\norm{B_{n,m}^{p,q}(\bfa\vert \bfb)[\bfc,\varphi]}_\infty \leq C\norm{\varphi}_1 \prod_{i=1}^{q}\norm{c_i'}_\infty.
	\end{equation}
	Moreover, if $(\bfa, \bfb,\bfc)\in\rmC^1(\Ss)^{m+n+q}$,  there exists  a constant $C>0$ that depends only on $n,\,m,\, q$,   and $\norm{(\bfa',\bfb')}_\infty$ such that
	 for all $\varphi\in\rmC(\Ss)$ we have
	\begin{equation}\label{eq:Anmq_Bnmpq_A1}
\|A_{n,m}^{1,q}(\bfa\vert \bfb)[\bfc,\varphi]\|_{ \rmC^1}  \leq C\norm{\varphi}_{\rm\infty} \prod_{i=1}^{q}\norm{c_i}_{ \rmC^1}.
	\end{equation}
\end{Lemma}
\begin{proof}
	To start, we let $K_A$ and $K_B$ denote the kernels of the integral operators~$A_{n,m}^{\ell,q}$ and~$B_{n,m}^{p,q}$, respectively, that is,
	\begin{equation}\label{eq:KA_KB}
	\begin{aligned}
		A_{n,m}^{\ell,q}(\bfa\vert \bfb)[\bfc,\varphi](\xi)&=\int_{-\pi}^\pi K_A(\xi,s)\varphi(\xi-s)\dx{s},\\
		B_{n,m}^{p,q}(\bfa\vert \bfb)[\bfc,\varphi](\xi)&=\int_{-\pi}^\pi K_B(\xi,s)\varphi(\xi-s)\dx{s}, \qquad \xi\in\RR.
	\end{aligned}
	\end{equation}
	We begin by establishing \eqref{eq:Anmq_Bnmpq_inf} for $B_{n,m}^{p,q}$.  
	Since $p\geq1$ and  $n+q+1\geq p$,  we infer from~\eqref{eq:deriv} that
	\begin{align*}
		|K_B(\xi,s)|&\leq\Bigg(\prod_{i=1}^{n}\norm{b_i'}_\infty\Bigg)\Bigg(\prod_{i=1}^{q}\norm{c_i'}_\infty\Bigg)
	 |s|^{p-1}\abs{\frac{s/2}{\tss}}^{n+q+1-p}\leq C\prod_{i=1}^{q}\norm{c_i'}_\infty
	\end{align*}	
	for $ \xi\in\RR$ and  $0\neq s \in(-\pi,\pi),$ which proves \eqref{eq:Anmq_Bnmpq_inf} for $B_{n,m}^{p,q}$.
	
	Since $A_{n,m}^{\ell,q}$ is linear in $c_i$, $1\leq i\leq q$, it suffices to establish the estimate \eqref{eq:Anmq_Bnmpq_inf} for  $A_{n,m}^{\ell,q}$ under the assumption that  $\|\bfc'\|_\infty\leq 1$. 
	Let $F:\RR^{n+q+m}\to\RR$ be the locally Lipschitz continuous function defined by
	\begin{equation}\label{eq:def_F}
F(x,y,z)=\frac{1}{2\pi}\frac{\bigg(\prod\limits_{i=1}^{n}x_i\bigg)\bigg(\prod\limits_{i=1}^{q}y_i\bigg)}{\prod\limits_{i=1}^{m}(1+z_i^2)} \qquad \text{for $(x,y,z)\in\RR^{n+q+m}$.}
	\end{equation}
Given $\xi\in\RR$, $s\neq 0$, and $\bfd=(d_1,\ldots, d_l) \in\Wpp{1}{\infty}{\Ss}^l$,   $l\in\NN$, we introduce the shorthand notation
\begin{equation}\label{eq:bfd}
\frac{\Tf{\xi,s}{\bfd}}{\tss}:=\Big(\frac{\Tf{\xi,s}{d_1}}{\tss},\ldots,\frac{\Tf{\xi,s}{d_l}}{\tss}\Big).
\end{equation}
 Together with \eqref{eq:deriv}, we may now estimate for $\xi\in\RR$ and $0\neq s\in(-\pi,\pi)$
	\begin{equation}\label{eq:K_split}
	\begin{aligned}
		|K_A(\xi,s)|&\leq C\bigg|\frac{1}{\tss^\ell}-\frac{1}{(s/2)^\ell}\bigg|\\[1ex]
		&\quad+\bigg|\frac{1}{(s/2)^\ell}\bigg|\, 
		\bigg|F\Big(\frac{\Tf{\xi,s}{\mathbf{b}}}{\tss}, \frac{\dg{\xi,s}{\bfc}/2}{\tss}, \frac{\Tf{\xi,s}{\mathbf{a}}}{\tss}\Big)-
		F\Big(\frac{\dg{\xi,s}{\bfb/2}}{s/2}, \frac{\dg{\xi,s}{\bfc}/2}{s/2}, \frac{\dg{\xi,s}{\bfa}/2}{s/2}\Big)\bigg|.
	\end{aligned}
	\end{equation}
In view of \eqref{eq:tanh} and \eqref{eq:tan}, we have
\begin{align*}
&\bigg|\frac{1}{\tss^\ell}-\frac{1}{(s/2)^\ell}\bigg|\leq 2|s|^{2-\ell},\qquad 0\neq s\in(-\pi,\pi),\, \ell\in\{1,\, 2\},\\
&\bigg|\frac{\Tf{\xi,s}{\mathbf{d}}}{\tss}-\frac{\dg{\xi,s}{\mathbf{d}/2}}{s/2}\bigg|
+\bigg|\frac{\dg{\xi,s}{\mathbf{d}}/2}{\tss}-\frac{\dg{\xi,s}{\mathbf{d}/2}}{s/2}\bigg|\leq C|s|^{2},\qquad 0\neq s\in(-\pi,\pi),\, \xi\in\RR,
\end{align*}
with $C$ depending only on $\|\mathbf{d}'\|_\infty.$
These estimates together with \eqref{eq:K_split} immediately imply that 
\begin{equation}\label{eq:KA_est}
\abs{K_A(\xi,s)}\leq C|s|^{2-\ell},\qquad 0\neq s\in(-\pi,\pi),\, \xi\in\RR,\, \ell\in\{1,\, 2\},
\end{equation}
and the desired estimate \eqref{eq:Anmq_Bnmpq_inf} for $ A_{n,m}^{\ell,q}$ follows.

Since for $\varphi\in\rmC(\Ss)$ the  continuity of parameter integrals implies that both $A_{n,m}^{\ell,q}(\bfa\vert \bfb)[\bfc,\varphi] $ and~${B_{n,m}^{p,q}(\bfa\vert \bfb)[\bfc,\varphi]}$ belong to~$ \rmC(\Ss)$, the density of $ \rmC(\Ss)$  in $\Lp{1}{\Ss} $ leads us to \eqref{eq:Anmq_Bnmpq_C}.

 It remains to establish \eqref{eq:Anmq_Bnmpq_A1}. To this end, we first assume that $\varphi\in\rmC^1(\Ss)$.
 Since
 \begin{align*}
 &K_A(\cdot, s)\varphi(\cdot-s)\in\rmC^1(\Ss),\qquad 0\neq s\in(-\pi,\pi),\\[1ex]
 &K_A(\xi,\cdot)\varphi(\xi-\cdot)\in\rmC^1([-\pi,\pi]), \qquad \xi\in\RR,
 \end{align*}
 Fubini's theorem and integration by parts imply that $ A_{n,m}^{1,q}(\bfa\vert \bfb)[\bfc,\varphi]$ is weakly differentiable with  
	\begin{equation*}
	\begin{aligned}
		&\big(A_{n,m}^{1,q}(\bfa\vert \bfb)[\bfc,\varphi]\big)'(\xi)\\
		&\qquad=\big(K_A(\xi,-\pi) -K_A(\xi,\pi)\big)\varphi(\xi-\pi)+\int_{-\pi}^\pi\big(\partial_\xi K_A+\partial_s K_A\big)(\xi,s)\varphi(\xi-s)\dx{s} 
	\end{aligned}
	\end{equation*}
 for  $\xi\in\RR$, hence, using the notation
 	\begin{equation}\label{eq:bfb_j}
	\begin{aligned}
		\bfb_j&:=( b_1,\dots,b_{j-1},b_{j+1},\dots,b_n),\qquad  1\leq j\leq n,\\
		\bfc_j&:=( c_1,\dots,c_{j-1},c_{j+1},\dots,c_q),\qquad 1\leq j\leq q,
	\end{aligned}
	\end{equation} 
	we have
 	\begin{equation}\label{eq:deriv_A}
	\begin{aligned}
		\big(A_{n,m}^{1,q}(\bfa\vert \bfb)[\bfc,\varphi]\big)'&=2\frac{1-(-1)^{n+q+1}}{2\pi}
		\frac{\prod_{i=1}^{n}(\dg{\cdot,\pi}{b_i}/\pi)\prod_{i=1}^{q}(\dg{\cdot,\pi}{c_i}/\pi)}{\prod_{i=1}^{m}\sqp{1+(\dg{\cdot,\pi}{a_i}/\pi)^2}}\frac{\varphi(\cdot-\pi)}{\pi}\\
		&\quad+\sum_{j=1}^{n}\frac{b_j'}{2}\Big(A_{n-1,m}^{2,q}(\bfa\vert \bfb_j)[\bfc,\varphi]-B_{n+1,m}^{1,q}(\bfa\vert \bfb,b_j)[\bfc,\varphi]\Big)\\
		&\quad+\sum_{j=1}^{q}\frac{c_j'}{2}A_{n,m}^{2,q-1}(\bfa\vert \bfb)[\bfc_j,\varphi]\\
		&\quad-\frac{n+q+1}{2}\Big(A_{n,m}^{2,q}(\bfa\vert \bfb)[\bfc,\varphi] +B_{n,m}^{1,q}(\bfa\vert \bfb)[\bfc,\varphi]\Big)\\
		&\quad+\sum_{j=1}^{m} \Big[A_{n+2,m+1}^{2,q}(\bfa,a_j\vert \bfb,a_j,a_j)[\bfc,\varphi]+B_{n+2,m+1}^{1,q}(\bfa,a_j\vert\bfb,a_j,a_j)[\bfc,\varphi]\\
		&\hspace{4em}-a_j'A_{n+1,m+1}^{2,q}(\bfa,a_j\vert\bfb,a_j)[\bfc,\varphi]\\
		&\hspace{4em}+a_j'B_{n+3,m+1}^{1,q}(\bfa,a_j\vert \bfb ,a_j,a_j,a_j)[\bfc,\varphi]\Big].
	\end{aligned}
	\end{equation}
	The remaining claim \eqref{eq:Anmq_Bnmpq_A1} follows now from the previous relation and \eqref{eq:Anmq_Bnmpq_C} in view of the density of~$\rmC^1(\Ss)$ in~$\rmC(\Ss)$.
\end{proof}

We next study the singular integral  operator  $B_{n,m}^{0,q}$.
\begin{Lemma}\label{Lem:Bnmq_L2_L2}
	Let $n,\,m,\, q\in\NN_0$ and let~$(\bfa, \bfb)\in\Wpp{1}{\infty}{\Ss}^{m+n}$ be given. 
	Then,	there exists a constant $C>0$ that depends only on $n,\,m,\,q$,  and $\norm{(\bfa',\bfb')}_\infty$ such that for all $\bfc\in\Wpp{1}{\infty}{\Ss}^{q}$ and~$\varphi\in\Lp{2}{\Ss}$ we have 
	\begin{equation}
		\norm{B_{n,m}^{0,q}(\bfa\vert \bfb)[\bfc,\varphi]}_2\leq C\norm{\varphi}_2 \prod_{i=1}^{q}\norm{c_i'}_\infty.
	\end{equation}
\end{Lemma}
\begin{proof} 
The claim is  a straightforward consequence of \eqref{eq:B=A+C}, Lemma~\ref{Lem:Anmq_Bnmpq_inf}, and Lemma~\ref{Lem:Cnm_est}~(i).
\end{proof}

\begin{Lemma}\label{Lem:Bnmq_Hr-1_Hr-1}	
Let $n,\,m,\, q\in\NN_0$, 	$r\in(3/2,2)$, and $(\bfa,\bfb)\in\Hp{r}{\Ss}^{m+n}$ be given. 
Then, there exists a constant~$C>0$ that depends only on $n,\,m,\, q,\,r$,   and~$\norm{(\bfa,\bfb)}_{\nHp{r}}$ 
	 	such that for all~${\bfc\in\Hp{r}{\Ss}^q}$ and $\varphi\in\Hp{r-1}{\Ss}$ we have
	 		\begin{equation}\label{eq:Bnmq_Hr-1_Hr-1}
			\norm{B_{n,m}^{0,q}(\bfa\vert \bfb)[\bfc,\varphi]}_{\nHp{r-1}}\leq C\norm{\varphi}_{\nHp{r-1}}\prod_{i=1}^{q}\norm{c_i}_{\nHp{r}}.
		\end{equation}
	\end{Lemma}
	\begin{proof}
	The claim is an immediate consequence of \eqref{eq:B=A+C},  Lemma~\ref{Lem:Cnm_est}~(iii), \eqref{eq:Anmq_Bnmpq_inf} and~\eqref{eq:deriv_A}.
	\end{proof}

The next result shows that the (singular) integral operators $B_{n,m}^{p,q}$ are locally Lipschitz continuous with respect to~$(\bfa,\bfb,\bfc).$
\begin{Lemma}\label{Lem:Bnmpq_loc_Lip}
		Given $n,\,m,\,p,\,q\in\NN_0$ with $p\leq n+q+1,$ we have
		\begin{align}
			&\big[(\bfa,\bfb,\bfc)\mapsto B_{n,m}^{0,q}(\bfa\vert\bfb)[\bfc,\cdot]\big]\in\rmC^{1-}\big(\Wpp{1}{\infty}{\Ss}^{m+n+q},\mcL(\Lp{2}{\Ss})\big),\label{eq:Bnm0q_loc_Lip}\\
			&\big[(\bfa,\bfb,\bfc)\mapsto B_{n,m}^{p,q}(\bfa\vert\bfb)[\bfc,\cdot]\big]\in\rmC^{1-}\big(\Wpp{1}{\infty}{\Ss}^{m+n+q},\mcL(\Lp{1}{\Ss}, \rmC(\Ss))\big),\qquad p\geq 1.\label{eq:Bnmpq_loc_Lip}
		\end{align}
	\end{Lemma}	
	\begin{proof}
		Given  $(\bfa,\bfb,\bfc),(\tilde\bfa,\tilde\bfb,\tilde\bfc)\in\Wpp{1}{\infty}{\Ss}^{m+n+q}$, $\varphi\in\rmC^\infty(\Ss)$, and $p\in\NN_0$, we have
		\begin{equation*}
		\begin{aligned}
			&B_{n,m}^{p,q}(\bfa\vert\bfb)[\bfc,\varphi]-B_{n,m}^{p,q}(\tilde\bfa\vert\tilde\bfb)[\tilde\bfc,\varphi]\\
			&= \sum_{i=1}^{q}B_{n,m}^{p,q}(\bfa\vert\bfb)[\tilde{c}_1,\dots,\tilde{c}_{i-1},c_i-\tilde{c}_i,c_{i+1},\dots,c_q,\varphi]\\
			&\quad+\sum_{i=1}^{n}\big(B_{n,m}^{p,q}(\bfa\vert \tilde{b}_1,\dots,\tilde{b}_{i-1},b_i,\dots,b_n)-B_{n,m}^{p,q}(\bfa\vert \tilde{b}_1,\dots,\tilde{b}_i,b_{i+1},\dots,b_n)\big)[\tilde\bfc,\varphi]\\
			&\quad+\sum_{i=1}^{m}\big(B_{n+2,m+1}^{p,q}(\tilde{a}_1,\dots,\tilde{a}_i,a_i,\dots,a_m\vert\tilde\bfb,\tilde{a}_i,\tilde{a}_i)\\
			&\hspace{4em}-B_{n+2,m+1}^{p,q}(\tilde{a}_1,\dots,\tilde{a}_i,a_i,\dots,a_m\vert\tilde\bfb,\tilde{a}_i,a_i)\big)[\tilde\bfc,\varphi]\\
			&\quad+\sum_{i=1}^{m}\big(B_{n+2,m+1}^{p,q}(\tilde{a}_1,\dots,\tilde{a}_i,a_i,\dots,a_m\vert\tilde\bfb,\tilde{a}_i,a_i)\\
			&\hspace{4em}-B_{n+2,m+1}^{p,q}(\tilde{a}_1,\dots,\tilde{a}_i,a_i,\dots,a_m\vert\tilde\bfb,a_i,a_i)\big)[\tilde\bfc,\varphi]. 
		\end{aligned}
		\end{equation*}
		The first term on the right-hand side may be estimated by using Lemma~\ref{Lem:Bnmq_L2_L2} if~$p=0$ and Lemma~\ref{Lem:Anmq_Bnmpq_inf} for~$p\geq 1$.
		For the remaining terms it suffices to show that, given~${d,\,\tilde{d}\in\Wpp{1}{\infty}{\Ss}}$, we have
		\begin{equation}\label{eq:B0_loc_Lip}
			\|B_{n+1,m}^{0,q}(\bfa\vert\bfb,d)[\bfc,\varphi]-B_{n+1,m}^{0,q}(\bfa\vert\bfb,\tilde{d} )[\bfc,\varphi]\|_2\leq C\|d-\tilde{d} \|_{\nWpp{1}{\infty}}\norm{\varphi}_2,
		\end{equation}
	for \eqref{eq:Bnm0q_loc_Lip}, respectively,
	 \begin{equation} \label{eq:Bp_loc_Lip}
			\|B_{n+1,m}^{p,q}(\bfa\vert\bfb,d)[\bfc,\varphi]-B_{n+1,m}^{p,q}(\bfa\vert\bfb,\tilde{d} )[\bfc,\varphi]\|_\infty\leq C\|d-\tilde{d}\|_{\nWpp{1}{\infty}}\norm{\varphi}_1,\qquad p\geq1,
		\end{equation}
	for \eqref{eq:Bnmpq_loc_Lip}, with a constant $C$ that depends only on $\|(\bfa,\bfb,\bfc, d,\tilde{d})\|_{\nWpp{1}{\infty}} $ and $n,\, m,\, p,\, q.$ 
		
		To show \eqref{eq:B0_loc_Lip}-\eqref{eq:Bp_loc_Lip}, we infer from the fundamental theorem of calculus that
		\begin{equation}\label{eq:x-tanh_Lip}
			\abs{(x-\tanh(x))-(y-\tanh(y))}\leq(x^2+y^2)|x-y|,\qquad x,\,y\in\RR.
		\end{equation}
 With $F:\RR^{n+q+m}\to\RR$ denoting the smooth function defined in \eqref{eq:def_F} we then compute, by using also the notation \eqref{eq:bfd}, that
		\begin{equation*}
		\begin{aligned}
			&\hspace{-0.5cm}\big(B_{n+1,m}^{p,q}(\bfa\vert\bfb ,d)[\bfc,\varphi]-B_{n+1,m}^{p,q}(\bfa\vert\bfb ,\tilde{d} )[\bfc,\varphi]\big)(\xi)\\
			&=\PV\int_{-\pi}^{\pi}F\left(\frac{\T{\xi,s}\bfb}{\tss},\frac{\dg{\xi,s}{\bfc}/2}{\tss},\frac{\T{\xi,s}\bfa}{\tss}\right)\frac{\T{\xi,s}d-\T{\xi,s}\tilde{d} }{\tss}\frac{\varphi(\xi-s)}{\tss^{1-p}}\dx{s}\\
			&=B_{n,m}^{p,q+1}(\bfa\vert\bfb)[\bfc,d-\tilde{d} ,\varphi](\xi)-\int_{-\pi}^{\pi}K(\xi,s)\varphi(\xi-s)\dx{s}
		\end{aligned}
		\end{equation*}
		for $\xi\in\RR$ and $p\in\NN_0$, where, given $\xi\in\RR$ and $0\neq s\in(-\pi,\pi)$, we set
		\[
K(\xi,s):=F\left(\frac{\T{\xi,s}\bfb}{\tss},\frac{\dg{\xi,s}{\bfc}/2}{\tss},\frac{\T{\xi,s}\bfa}{\tss}\right)
			\frac{(\dg{\xi,s}d/2-\T{\xi,s}d)-(\dg{\xi,s}\tilde{d}/2-\T{\xi,s}\tilde{d}) }{\tss^{2-p}}.		
		\]
		The function $B_{n,m}^{p,q+1}(\bfa\vert\bfb)[\bfc,d-\tilde{d} ,\varphi]$ may be estimated by using Lemma~\ref{Lem:Bnmq_L2_L2} if~$p=0$ 
		and Lemma~\ref{Lem:Anmq_Bnmpq_inf} for~$p\geq 1$, and we are left to estimate the integral term.
		To this end we rely on~\eqref{eq:deriv} and~\eqref{eq:x-tanh_Lip} to obtain that $|K(\xi,s)|\leq C\|d-\tilde{d}\|_\infty$ for all $ \xi\in\RR$ and $0\neq s\in(-\pi,\pi)$, and therefore 
		\begin{align*}
		\abs{\int_{-\pi}^{\pi}K(\xi,s)\varphi(\xi-s)\dx{s}}\leq C\|d-\tilde{d}\|_\infty\|\varphi\|_1\leq C\|d-\tilde{d}\|_\infty\|\varphi\|_2,\qquad \xi\in\RR.
        \end{align*}	
        This completes the proof.	 
	\end{proof}

	Using Lemma~\ref{Lem:Bnmpq_loc_Lip}, we next prove that the operators $B_{n,m}^{p,q}$, with $p\geq 1$,  have a regularizing effect.
	\begin{Lemma}\label{Lem:Bnmpq_H1_L2}
		Let $n,\, m,\,p,\,\,q\in\NN_0$, ${1\leq p\leq n+q+1}$,  and~${(\bfa,\bfb,\bfc)\in\rmC^1(\Ss)^{m+n+q}}$.
		Then,   there exists a constant $C>0$ that depends only on $n,\,m,\,p,\,q,$ and $ \norm{(\bfa',\bfb')}_{\infty}$  such that for all~${\varphi\in\Lp{2}{\Ss}}$ we have
		\begin{equation}\label{eq:Bnmpq_H1_L2}
			\norm{B_{n,m}^{p,q}(\bfa\vert \bfb)[\bfc,\varphi]}_{\nHp{1}}\leq C\norm{\varphi}_{2}\prod_{i=1}^{q}\norm{c_i'}_{\infty}.
		\end{equation}
		Furthermore, given $r\in(3/2,2)$ and $(\bfa,\bfb,\bfc)\in\Hp{r}{\Ss}^{m+n+q}$, there exists a 
		constant $C>0$ that depends only on $n,\,m,\,p,\,q,$ and $ \norm{(\bfa,\bfb)}_{\nHp{r}}$ such that for all~$\varphi\in\Hp{r-1}{\Ss}$ we have
		\begin{equation}\label{eq:Bnmpq_Hr_Hr-1}
			\norm{B_{n,m}^{p,q}(\bfa\vert \bfb)[\bfc,\varphi]}_{\nHp{r}}\leq C\norm{\varphi}_{\nHp{r-1}}\prod_{i=1}^{q}\norm{c_i}_{\nHp{r}}.
		\end{equation}
	\end{Lemma}
	\begin{proof}
We first assume that ${(\bfa,\bfb,\bfc,\varphi)\in\rmC^\infty(\Ss)^{m+n+q+1}}.$		
Recalling the  notation~\eqref{eq:KA_KB},
the theorem on the differentiation of parameter  integrals ensures that $B_{n,m}^{p,q}(\bfa\vert \bfb)[\bfc,\varphi]$ is continuously differentiable with 
\begin{equation*}
\big(B_{n,m}^{p,q}(\bfa\vert \bfb)[\bfc,\varphi]\big)'(\xi)=\int_{-\pi}^\pi \partial_\xi K_B(\xi,s)\varphi(\xi-s)-K_B(\xi,s)\partial_s (\varphi(\xi-s)) \dx{s}, \qquad \xi\in\RR.
\end{equation*}
Using integration by parts, we then get
		\begin{equation*} 
		\begin{aligned}
			&(B_{n,m}^{p,q}(\bfa\vert \bfb)[\bfc,\varphi])'\\
			&=\frac{1}{2}\sum_{j=1}^{n} b_j'\big(B_{n-1,m}^{p-1,q}(\bfa\vert \bfb_j)-B_{n+1,m}^{p+1,q}(\bfa\vert \bfb,b_j)\big)[\bfc,\varphi]
			+\frac{1}{2}\sum_{j=1}^{q}c_j'B_{n,m}^{p-1,q-1}(\bfa\vert \bfb)[\bfc_j,\varphi]\\
			&\quad+\sum_{j=1}^{m}a_j'\big(B_{n+3,m+1}^{p+1,q}(\bfa,a_j\vert \bfb,a_j,a_j,a_j)-B_{n+1,m+1}^{p-1,q}(\bfa,a_j\vert \bfb,a_j)\big)[\bfc,\varphi]\\
			&\quad+\sum_{j=1}^{m}\big(B_{n+2,m+1}^{p-1,q}(\bfa,a_j\vert \bfb,a_j,a_j)+B_{n+2,m+1}^{p+1,q}(\bfa,a_j\vert \bfb,a_j,a_j)\big)[\bfc,\varphi]\\
			&\quad+\frac{p-n-q-1}{2}\big(B_{n,m}^{p-1,q}(\bfa\vert \bfb)+B_{n,m}^{p+1,q}(\bfa\vert \bfb)\big)[\bfc,\varphi],
		\end{aligned}
		\end{equation*}
		with the observation that the last term is meaningful only if $1\leq p\leq n+q,$ otherwise it is not present in the formula above.
		The functions $\bfb_j$, $1\leq j\leq n$, and $\bfc_j$, $1\leq j\leq q$, are as defined in \eqref{eq:bfb_j}.
		A standard density argument together with Lemma~\ref{Lem:Bnmpq_loc_Lip} ensures now that~${B_{n,m}^{p,q}(\bfa\vert \bfb)[\bfc,\varphi]\in\nHp{1}(\Ss)}$ for all ${(\bfa,\bfb,\bfc)\in\rmC^1(\Ss)^{m+n+q}}$
		 and ${\varphi\in\Lp{2}{\Ss}}$, the estimate~\eqref{eq:Bnmpq_H1_L2} being a direct consequence of  Lemma~\ref{Lem:Anmq_Bnmpq_inf} and Lemma~\ref{Lem:Bnmq_L2_L2} while the estimate~\eqref{eq:Bnmpq_Hr_Hr-1} follows from Lemma~\ref{Lem:Anmq_Bnmpq_inf}, Lemma~\ref{Lem:Bnmq_Hr-1_Hr-1}, and \eqref{eq:Bnmpq_H1_L2}.
	\end{proof}

	 The next result shows that the operator $B_0$ defined in \eqref{eq:B0_alt}  has similar regularity properties as $B_{n,m}^{p,q}$ with $1\leq p\leq n+q+1$, see \eqref{eq:Bnmpq_Hr_Hr-1}.
	\begin{Lemma}\label{Lem:Map_B0}
		Let $r\in (3/2,2)$. Given $f\in\Hp{r}{\Ss}$, there exists a constant $C>0$ that depends only on~$\|f\|_{\nHp{r}}$ such that for all~${\varphi\in  \Hp{r-1}{\Ss}}$ we have
		\begin{equation}\label{eq:B0_Hs-1_Hs}
			\norm{B_0(f)[\varphi]}_{\nHp{r}}\leq C\norm{\varphi}_{\nHp{r-1}}.
		\end{equation}
	\end{Lemma}
	\begin{proof}
	 We first prove that if $f\in \Wpp{1}{\infty}{\Ss}$, then $B_0(f)\in\mcL(\Lp{2}{\Ss}, \Lp{\infty}{\Ss})$.
	Indeed,  using~\eqref{eq:GPi} and the fact that $\ln(\sin(\cdot/2)^2)\in\Lp{2}{\Ss}$, we deduce, in view of the inequality
		\[
		\Big|\ln\big(\sin^2(s/2)+\sinh^2(\dg{\xi, s}{f}/2))\big)\Big|\leq \Big|\ln\big(\sin^2(s/2) \big)\Big|+\ln \big(1+\sinh^2( \pi\|f'\|_\infty)\big),\quad \xi,\,s\in\Ss,
		\]
		 that for~${\varphi\in \Lp{2}{\Ss}}$ we have
		\begin{align*}
			|B_0(f)[\varphi](\xi)| &\leq \bigg(\int_{-\pi}^\pi\big[\big|\ln\!\big(\sin^2(s/2)\big)\big
			|+\ln \big(1+\sinh^2( \pi\|f'\|_\infty)\big) \big]\, |\varphi(\xi-s)|\dx{s} \bigg)  \leq C\norm{\varphi}_{2}.
		\end{align*}

		We now assume that $f\in\Hp{r}{\Ss}$ and $\varphi\in\rmC^\infty(\Ss)$. 
		Using the theorem on the differentiation of parameter integrals and subsequently integration by parts, we find that $B_0(f)[\varphi]$ is continuously differentiable and its derivative is given by
		\begin{align*}
			(B_0(f)[\varphi])'=f'B_2(f)[\varphi]+ B_1(f)[\varphi]\in \Hp{r-1}{\Ss},
		\end{align*}
		cf. Lemma~\ref{Lem:Bnmq_Hr-1_Hr-1} and Lemma~\ref{Lem:Bnmpq_H1_L2}.
		The claim follows now by a standard density argument in view of  Lemma~\ref{Lem:Bnmq_Hr-1_Hr-1} and Lemma~\ref{Lem:Bnmpq_H1_L2}.
	\end{proof}
	
	\subsection{Fr\'{e}chet differentiability}
	 This section is devoted to establishing   the following result.
	\begin{Corollary}\label{C:CCC} Given $r\in(3/2,2)$, the mappings 
	\begin{equation*}
\begin{aligned}
&[f\mapsto B_{n,m}^{0,q}(f)]: \Hp{r}{\Ss}\to \mcL(\Hp{r-1}{\Ss}),\\[1ex]
&[f\mapsto B_0(f)],\,[f\mapsto B_{n,m}^{p,q}(f)]: \Hp{r}{\Ss}\to \mcL(\Hp{r-1}{\Ss}, \Hp{r}{\Ss}),\qquad   1\leq p\leq n+q+1,\\[1ex]
\end{aligned}
\end{equation*}
are smooth.
	\end{Corollary}
	
 The proof of Corollary~\ref{C:CCC} is presented at the end of this section, as it requires some preparation.
Let us first note that Lemma~\ref{Lem:Bnmq_Hr-1_Hr-1} and Lemma~\ref{Lem:Bnmpq_H1_L2} ensure that the mappings defined above are well-defined.
In order to establish the smoothness property, we further introduce the operators
\begin{equation}\label{eq:Bk_map_diff}
\begin{aligned}
		&B_{n,m}^{0,q,k}: \Hp{r}{\Ss}\to\mcL^{k}_{sym}(\Hp{r}{\Ss},\mcL(\Hp{r-1}{\Ss})),\\[1ex]
		&   B_{n,m}^{p,q,k}: \Hp{r}{\Ss}\to\mcL^{k}_{sym}(\Hp{r}{\Ss},\mcL(\Hp{r-1}{\Ss}, \Hp{r}{\Ss})),\qquad 1\leq p\leq n+q+k+1,
\end{aligned}
	\end{equation}
	by
	\begin{equation*}
		B_{n,m}^{p,q,k}(f)[f_1,\dots,f_k][\cdot]:= B_{n,m}^{p,q+k}(f,\dots,f\vert f,\dots,f)[f,\dots,f,f_1,\dots,f_k,\cdot].
	\end{equation*}
Let us note that $B_{n,m}^{p,q}(f)=B_{n,m}^{p,q,0}(f)$. 
The next lemma is the main step towards proving the  smoothness property for the  operators~$B_{n,m}^{p,q}$.

\begin{Lemma}\label{Lem:Frechet_Bnmpq}
		The mappings \eqref{eq:Bk_map_diff} are Fr\'{e}chet differentiable. Moreover, the Fr\'{e}chet derivative $\partial B_{n,m}^{p,q,k}(f_0)$ is given by 
		\begin{equation}\label{eq:Frechet_Bnmpq}
		\begin{aligned}
			\partial B_{n,m}^{p,q,k}(f_{0})[f][f_1,\dots,f_k]&=n\big(B_{n-1,m}^{p,q,k+1}(f_{0}) - B_{n+1,m}^{p+2,q,k+1}(f_{0})\big)[f_1,\dots,f_k,f]\\
			&\quad +2m\big(B_{n+3,m+1}^{p+2,q,k+1}(f_{0})-B_{n+1,m+1}^{p,q,k+1}(f_0)\big)[f_1,\dots,f_k,f]\\
			&\quad+qB_{n,m}^{p,q-1,k+1}(f_0)[f_1,\dots,f_k,f]
		\end{aligned}
		\end{equation}
		 for $ f_0,\,f,\, f_1,\ldots,f_k\in \Hp{r}{\Ss}$, where terms with negative indices are to be neglected.  
	\end{Lemma}
	\begin{proof}
		Defining $\phi:=\phi_{n,m}^{p,q}$ by the formula
		\begin{equation*}
			\phi(\eta,s):=\frac{1}{2\pi} \frac{\p{\frac{\tanh(\eta)}{\ts{s}}}^{n}\p{\frac{\eta}{\ts{s}}}^q}{\sqp{1+\p{\frac{\tanh(\eta)}{\ts{s}}}^2}^{m}}\tss^p,\qquad \text{$\eta\in\RR,\, 0\neq s\in(-\pi,\pi)$,}
		\end{equation*}
we have for $\xi\in\RR$,   $f,\, f_1,\ldots,f_k\in \Hp{r}{\Ss}$,   and~${\varphi\in  \Hp{r-1}{\Ss}}$
		\begin{equation}
			B_{n,m}^{p,q,k}(f)[f_1,\dots,f_k][\varphi](\xi)=\PV\int_{-\pi}^{\pi}\p{\prod_{i=1}^{k}\frac{\dg{\xi,s}{f_i}/2}{\tss}}\phi(\dg{\xi,s}{f}/2,s)\frac{\varphi(\xi-s)}{\tss}\dx{s},
		\end{equation}
		the $\PV$ being needed only when $p=0$.
	Our goal is to prove that 
		\begin{equation}\label{eq:dB_phi}
		\begin{aligned}
			&\partial B_{n,m}^{p,q,k}(f_0)[f][f_1,\dots,f_k][\varphi](\xi)\\
			&=\PV\int_{-\pi}^{\pi}\p{\prod_{i=1}^{k}\frac{\dg{\xi,s}{f_i}/2}{\tss}}(\dg{\xi,s}{f}/2)\partial_\eta\phi(\dg{\xi,s}{f_0}/2,s)\frac{\varphi(\xi-s)}{\tss}\dx{s}
		\end{aligned}		
		\end{equation}
		for $\xi\in\RR$,   $f_0,\,f,\, f_1,\ldots,f_k\in \Hp{r}{\Ss}$,   and~${\varphi\in  \Hp{r-1}{\Ss}}$, 
		as  straightforward  computations show that the formulas~\eqref{eq:Frechet_Bnmpq} and~\eqref{eq:dB_phi} are equivalent. 
		
		Using Taylor's formula, we compute
\begin{equation}\label{eq:B_deriv_int}		
		\begin{aligned}
		&\hspace{-0.25cm}\big(B_{n,m}^{p,q,k}(f_0+f)-B_{n,m}^{p,q,k}(f_0)-\partial B_{n,m}^{p,q,k}(f_0)[f]\big)[f_1,\dots,f_k][\varphi](\xi)\\
		&\hspace{-0.25cm}=\PV\int_{-\pi}^{\pi} \p{\prod_{i=1}^{k}\frac{\dg{\xi,s}{f_i}/2}{\tss}}(\dg{\xi,s}{f}/2)^2\bigg(\int_0^1 (1-\tau)\partial_\eta^2\phi(\delta_{[\xi,s]}f_\tau/2,s)\dx{\tau}\bigg)\frac{\varphi(\xi-s)}{\tss}\dx{s},
		\end{aligned}
		\end{equation}
		where $ f_\tau:= f_0+\tau f$ for $\tau\in[0,1]$, and $\partial_\eta^2\phi=\partial_\eta^2\phi_{n,m}^{p,q}$ is given by
		\begin{equation}\label{eq:D2_phi}
		\begin{aligned}
		\partial_\eta^2\phi_{n,m}^{p,q}&=\frac{1}{\tss^2}\Big\{n(n-1)\phi_{n-2,m}^{p,q}+2nq\phi_{n-1,m}^{p,q-1}+q(q-1)\phi_{n,m}^{p,q-2}-2m(2n+1)\phi_{n,m+1}^{p,q}\\[1ex]
		&\hspace{1.25cm} -2nq\phi_{n+1,m}^{p+2,q-1}-2n^2\phi_{n,m}^{p+2,q}+8m(n+1)\phi_{n+2,m+1}^{p+2,q}+n(n+1)\phi_{n+2,m}^{p+4,q}\\[1ex]
		&\hspace{1.25cm}-2m(2n+3)\phi_{n+4,m+1}^{p+4,q}+4mq\phi_{n+3,m+1}^{p+2,q-1}-4mq\phi_{n+1,m+1}^{p,q-1}\\[1ex]
		&\hspace{1.25cm}+4m(m+1)\phi_{n+6,m+2}^{p+4,q}-8m(m+1)\phi_{n+4,m+2}^{p+2,q}+4m(m+1)\phi_{n+2,m+2}^{p,q}\Big\}
		\end{aligned}
		\end{equation}
		in $\RR\times \big((-\pi,\pi)\setminus\{0\}\big)$ and for all $0\leq p\leq n+q+k+1$.
		Recalling \eqref{eq:deriv}, in all the terms on the right-hand side of \eqref{eq:B_deriv_int} where $\phi_{n,m}^{p,q}$ with $p\geq 1$ 
		appear, the $\PV$ is not needed and we may interchange the order of integration by using Fubini's theorem. 
		
		Assume first that $p\geq 1$.  We then infer  from \eqref{eq:B_deriv_int} and \eqref{eq:D2_phi}, 
		after interchanging the order of integration in the last line of~\eqref{eq:B_deriv_int}, that 
		\begin{equation}\label{eq:DB_by_B}
		\begin{aligned}
		&\big(B_{n,m}^{p,q,k}(f_0+f)-B_{n,m}^{p,q,k}(f_0)-\partial B_{n,m}^{p,q,k}(f_0)[f]\big)[f_1,\dots,f_k][\varphi]\\
		&=\int_0^1  (1-\tau)\Big\{ n(n-1) B_{n-2,m}^{p,q,k+2}+2nq B_{n-1,m}^{p,q-1,k+2}+q(q-1) B_{n,m}^{p,q-2,k+2}\\
		&\hspace{1.25cm}-2m(2n+1) B_{n,m+1}^{p,q,k+2} -2nqB_{n+1,m}^{p+2,q-1,k+2}-2n^2B_{n,m}^{p+2,q,k+2}\\
		&\hspace{1.25cm}+8m(n+1)B_{n+2,m+1}^{p+2,q,k+2}+n(n+1)B_{n+2,m}^{p+4,q,k+2}-2m(2n+3)B_{n+4,m+1}^{p+4,q,k+2}\\[1ex]
		&\hspace{1.25cm}+4mqB_{n+3,m+1}^{p+2,q-1,k+2}-4mqB_{n+1,m+1}^{p,q-1,k+2}+4m(m+1)B_{n+6,m+2}^{p+4,q,k+2}\\[1ex]
		&\hspace{1.25cm} -8m(m+1)B_{n+4,m+2}^{p+2,q,k+2}+4m(m+1)B_{n+2,m+2}^{p,q,k+2}\Big\}(f_\tau)[f_1,\ldots,f_k,f,f][\varphi]\dx{\tau}.
		\end{aligned}
		\end{equation}
		 Moreover,   Lemma~\ref{Lem:Bnmpq_H1_L2}  implies there exists a constant $C$ such that  for all~${\|f\|_{\nHp{r}}\leq 1 }$  we have
		\begin{equation*}
		\begin{aligned}
			&\big\|\big(B_{n,m}^{p,q,k}(f_0+f)-B_{n,m}^{p,q,k}(f_0)-\partial B_{n,m}^{p,q,k}(f_0)[f]\big)[f_1,\dots,f_k]\big\|_{\mcL\p{\Hp{r-1}{\Ss},\Hp{r}{\Ss}}}\\
			&\leq C\norm{f}_{\nHp{r}}^2\prod_{i=1}^{k}\norm{f_i}_{\nHp{r}},
		\end{aligned}
		\end{equation*}
				which proves \eqref{eq:Frechet_Bnmpq} for $p\geq 1$.
		
		Let now $p=0$. In this case,  the formula~\eqref{eq:DB_by_B} is still valid (and defines a function in~$\Hp{r-1}{\Ss}$).
		This formula is obtained again by interchanging the order of integration in~\eqref{eq:B_deriv_int} via \eqref{eq:D2_phi}, but slightly more subtle arguments are needed when considering the 
		terms of \eqref{eq:D2_phi} with $p=0$ as the $\PV$ symbol appears in front of the first integral in \eqref{eq:B_deriv_int}.
		More precisely, letting 
         \[		
		I(\xi,s,\tau):=\p{\prod_{i=1}^{k}\frac{\dg{\xi,s}{f_i}/2}{\tss}}(\dg{\xi,s}{f}/2)^2  (1-\tau)\partial_\eta^2\phi(\delta_{[\xi,s]}f_\tau/2,s)\frac{\varphi(\xi-s)}{\tss}
		\]
		denote the integrand in \eqref{eq:B_deriv_int}, it holds that 
		\begin{align*}
	    &\PV\int_{-\pi}^{\pi} \Big(\int_0^1 I(\xi,s,\tau)\dx{\tau}\Big)\dx{s}=\int_0^\pi \Big(\int_0^1 I(\xi,s,\tau)+I(\xi,-s,\tau)\dx{\tau}\Big)\dx{s}\\
		&=\int_0^1 \Big(\int_0^\pi I(\xi,s,\tau)+I(\xi,-s,\tau)\dx{s}\Big)\dx{\tau}=\int_0^1\Big(\PV\int_{-\pi}^{\pi} I(\xi,s,\tau)\dx{s}\Big)\dx{\tau},
		\end{align*}
	 by Fubini's theorem and in view of the estimate
	\[
\big|I(\xi,s,\tau)+I(\xi,-s,\tau)\big|	\leq \frac{C}{|s|^{5/2-r}},\qquad \xi\in\RR, \, 0\neq s\in(-\pi,\pi),\, \tau\in[0,1]. 
	\]
Applying Lemma~\ref{Lem:Bnmq_Hr-1_Hr-1} and  Lemma~\ref{Lem:Bnmpq_H1_L2}, we  conclude from \eqref{eq:DB_by_B} that there exists a constant~$C>0$ such that  for all~${\|f\|_{\nHp{r}}\leq 1 }$ we have
		\begin{equation*}
		\begin{aligned}
			&\big\|\big(B_{n,m}^{p,q,k}(f_0+f)-B_{n,m}^{p,q,k}(f_0)-\partial B_{n,m}^{p,q,k}(f_0)[f]\big)[f_1,\dots,f_k]\big\|_{\mcL\p{\Hp{r-1}{\Ss}}}\\
			&\leq C\norm{f}_{\nHp{r}}^2\prod_{i=1}^{k}\norm{f_i}_{\nHp{r}},
		\end{aligned}
		\end{equation*}
which proves the claim for $p=0.$
	\end{proof}
	We now show the Fr\'{e}chet differentiability of the operator $B_0$ defined in \eqref{eq:B0_alt}.
	\begin{Lemma}\label{Lem:Frechet_B0}
		Given $r\in(3/2,2)$, the map $B_0:\Hp{r}{\Ss}\to\mcL(\Hp{r-1}{\Ss},\Hp{r}{\Ss})$ is Fr\'{e}chet differentiable  and the Fr\'{e}chet derivative $\partial B_0(f_0)$ is given by 
		\begin{equation}\label{eq:Frechet_B0}
			\partial B_0(f_0)[f]=2B_{1,1}^{1,0,1}(f_0)[f]+2B_{1,1}^{3,0,1}(f_0)[f],\qquad f_0,\,f\in \Hp{r}{\Ss}.
		\end{equation} 
	\end{Lemma}	
	\begin{proof}
		We apply the same strategy as in the proof of Lemma~\ref{Lem:Frechet_Bnmpq}. Defining $\phi$ by the formula
		\begin{equation*}
			\phi(\eta,s):=\frac{1}{2\pi}\ln\!\p{\frac{\tss^2+\tanh^2(\eta)}{(1+\tss^2)(1-\tanh^2(\eta))}},\qquad 0\neq \eta\in\RR,\quad  s\in(-\pi,\pi),
		\end{equation*}
		we have
		\begin{equation}\label{eq:phi_deriv}
			\partial_\eta\phi(\eta,s)=\frac{1}{\pi}\frac{(1+\tss^2)\tanh(\eta)}{\tss^2+\tanh^2(\eta)},\qquad \partial_\eta^2\phi(\eta,s)
			=\frac{1}{\pi}\frac{(1+\tss^2)(1-\tanh^2(\eta))(\tss^2-\tanh^2(\eta))}{(\tss^2+\tanh^2(\eta))^2}.
		\end{equation}
		We prove that
		\begin{equation}\label{eq:DB0_by_phi}
			\partial B_0(f_0)[f][\varphi](\xi)=\int_{-\pi}^{\pi}(\dxsf /2)\partial_\eta\phi(\dg{\xi,s}{f_0 /2},s)\varphi(\xi-s)\dx{s},
		\end{equation}	
	since easy calculations show that \eqref{eq:Frechet_B0} and \eqref{eq:DB0_by_phi} coincide. Using Taylor's formula, Fubini's theorem, \eqref{eq:phi_deriv}, and \eqref{eq:DB0_by_phi}, we compute for $\xi\in\RR$, $f_0,\,f\in \Hp{r}{\Ss}$, 
	and $\varphi\in \Hp{r-1}{\Ss}$ that
	\begin{equation*}
	\begin{aligned}
		&\big(B_0(f_0+f)-B_0(f_0)-\partial B_0(f_0)[f]\big)[\varphi](\xi)\\
		&\quad= \int_{-\pi}^\pi(\dxsf/2)^2\bigg(\int_0^1 (1-\tau)\partial_\eta^2\phi (\dg{\xi,s}{f_\tau/2},s)\dx{\tau}\bigg) \varphi(\xi-s)\dx{s}\\
		&\quad=\int_0^1 (1-\tau)\bigg(\int_{-\pi}^\pi(\dxsf/2)^2\partial_\eta^2 \phi(\dg{\xi,s}{f_\tau/2},s)\varphi(\xi-s)\dx{s}\bigg)\dx{\tau}\\
		&\quad=2\int_0^1 (1-\tau)\Big\{B_{0,2}^{1,0,2}+B_{0,2}^{3,0,2}-B_{2,2}^{1,0,2}-2B_{2,2}^{3,0,2}\\
		&\hspace*{8em}-B_{2,2}^{5,0,2}+B_{4,2}^{3,0,2}+B_{4,2}^{5,0,2}\Big\}(f_\tau)[f,f][\varphi]\dx{\tau},
		\end{aligned}
	\end{equation*}
	where $f_\tau=f_0+\tau f$. 
	Using \eqref{eq:Bnmpq_Hr_Hr-1}, we thus find a constant $C>0$ such that for all $f\in\Hp{r}{\Ss}$ with $\norm{f}_{\nHp{r}}\leq 1$ we have
	\begin{equation*}
		\norm{B_0(f_0+f)-B_0(f_0)-\partial B_0(f_0)[f]}_{\mcL(\Hp{r-1}{\Ss},\Hp{r}{\Ss})}\leq C\norm{f}_{\nHp{r}}^2,
	\end{equation*}
	which proves the claim.
	\end{proof}	
	
		 We are now in a position to establish Corollary~\ref{C:CCC}.
	\begin{proof}[Proof of Corollary~\ref{C:CCC}]
	Recalling that $B_{n,m}^{p,q}(f)=B_{n,m}^{p,q,0}(f)$ for $f\in\Hp{r}{\Ss}$, the assertion is a direct consequence of Lemma~\ref{Lem:Frechet_Bnmpq} and Lemma~\ref{Lem:Frechet_B0}.
	\end{proof}

\section{Localization of the singular integral operators $C_{n,m}$}\label{Sec:B}
	
	In this section we show that the singular integral operators 
	 $C_{n,m}^0$  defined in \eqref{eq:Cnm0} can be locally approximated by Fourier multipliers, see Lemma~\ref{Lem:Cnm_approx_a} for the precise statement.
	As a starting point we infer from~\eqref{eq:Cnm} the following algebraic relations
		\begin{equation}\label{eq:Cnm_dif_a}
		\begin{aligned}
		&\big(C_{n,m}(\tilde\bfa) -C_{n,m}(\bfa)\big)[\bfb,\varphi]\\
&=\sum_{i=1}^{m}C_{n+2,m+1}(a_1,\dots,a_i,\tilde{a}_i,\dots,\tilde{a}_m)[\bfb,a_i+\tilde{a}_i,a_i-\tilde{a}_i,\varphi],\qquad  n\in\NN_0,\, m\in\NN,
		\end{aligned}
	\end{equation}
and 
		\begin{equation}\label{eq:dC-Cd}
		\begin{aligned}
		&d C_{n,m}(\bfa)[\bfb,\varphi] -C_{n,m}(\bfa)[\bfb,d\varphi]\\
	&=b_1C_{n,m}(\bfa)[b_2,\ldots,b_n,d,\varphi]-C_{n,m}(\bfa)[b_2,\ldots,b_n,d,b_1\varphi],\qquad n\in\NN,\,m\in\NN_0,
		\end{aligned}
	\end{equation}
	 which hold for  all   $\bfa,\,\tilde{\bfa}\in \Wpp{1}{\infty}{\Ss}^{m}$, $\bfb\in \Wpp{1}{\infty}{\Ss}^{n}, $ $d\in\Wpp{1}{\infty}{\Ss},$  
	 and~$\varphi\in\Lp{2}{\Ss}$.
	
	The following commutator property, see \cite[Lemma 12]{Abels.2022} for a similar result in a non-periodic setting, is an important tool in the analysis that follows.
	\begin{Lemma}\label{Lem:Cnm_com}
		Given $n,\,m\in\NN_0$  and $a,\,f\in\rmC^1(\Ss)$, there exists a constant $C>0$ that depends only on $n,\,m,\,r,$ and $\norm{(a,f)}_{\rmC^1}$  
		  such that for all $\varphi\in\Lp{2}{\Ss}$ we have
		\begin{equation}\label{eq:Cnm_com}
			\norm{a C_{n,m}^{0}(f)[\varphi]-C_{n,m}^{0}(f)[a\varphi]}_{\nHp{1}}\leq C\norm{\varphi}_{2}.
		\end{equation}
	\end{Lemma}
	\begin{proof}
 The proof is similar to that of \cite[Lemma 12]{Abels.2022}, and therefore we omit it.
	\end{proof}
	
	Let us now recall the definition of an $\ve$-partition of unity from Section~\ref{Sec:34}. 
	The central result in this section is the following lemma.
	\begin{Lemma}\label{Lem:Cnm_approx_a}
		Let $n,\,m\in\NN_0$, $3/2<r'<r<2$, $f\in\Hp{r}{\Ss}$, $a,\,b\in\Hp{r-1}{\Ss}$, and~${\eta>0}$ be given. 
		 Then, for any sufficiently small $\ve\in(0,1)$, there exists a constant~$K>0$ that depends on~$\ve,\,n,\,m,\,\|f\|_{\nHp{r}},$ and $\|(a,b)\|_{\nHp{r-1}}$ such that for all $1\leq j\leq N$ and $\varphi\in\Hp{r-1}{\Ss}$ we have 
		\begin{equation}\label{eq:Cnm_approx_a}
			\bigg\|\pi_j^\ve a C_{n,m}^0 (f)[b\varphi]-\frac{a(x_j^\ve)b(x_j^\ve)(f'(x_j^\ve))^n}{\big[1+(f'(x_j^\ve))^2\big]^m} H[\pi_j^\ve \varphi]\bigg\|_{\nHp{r-1}}
			\leq \eta \norm{\pi_j^\ve \varphi}_{\nHp{r-1}}+K\norm{\varphi}_{\nHp{r'-1}}.
		\end{equation}
	\end{Lemma}
	
	The proof of Lemma~\ref{Lem:Cnm_approx_a} relies  heavily on the result provided by the next lemma.
	\begin{Lemma}\label{Lem:chiC}
		Given $n,\,m\in\NN_0$, $3/2<r<2$, $\eta\in(0,\infty)$, and $f\in\Hp{r}{\Ss}$, for sufficiently  small $\ve\in(0,1)$ and all $1\leq j\leq N,$ $|y|\leq\ve$, and~${\varphi\in\Lp{2}{\Ss}}$ we have
		\begin{equation}\label{eq:chiC}
			\big\|T_j^\ve(f)[\tau_y (\pi_j^\ve \varphi)-\pi_j^\ve \varphi]\big\|_2 \leq \eta \norm{\tau_y (\pi_j^\ve \varphi)-\pi_j^\ve \varphi}_{2},
		\end{equation}
		where $T_j^\ve(f):=\chi_j^\ve C_{n+1,m}(f,\dots,f)[f,\dots,f,f-f'(x_j^\ve)\id_\RR,\cdot] $ and $\tau_y$ is defined in \eqref{eq:tauy}.
	\end{Lemma}
	\begin{proof}
Let $\ve\in(0,1) $. Since 
\[
T_j^\ve(f)[\tau_y(\pi_j^\ve \varphi)-\pi_j^\ve \varphi]
=\chi_j^\ve\big(C_{n+1,m}^0(f)-f'(x_j^\ve)C_{n,m}^0(f)\big)[\tau_y (\pi_j^\ve \varphi)-\pi_j^\ve \varphi]\in \Lp{2}{\Ss},
\]
we have by Lemma~\ref{Lem:Cnm_est}~(i) that
\begin{equation}\label{gf1}
\|T_j^\ve(f)[\tau_y (\pi_j^\ve \varphi)-\pi_j^\ve \varphi]\|_2=\|T_j^\ve(f)[\tau_y (\pi_j^\ve \varphi)-\pi_j^\ve \varphi]\|_{\Lp{2}{(x_j^\ve-\pi,x_j^\ve+\pi)}}.
\end{equation}
We now introduce the   Lipschitz continuous function $F_j:\RR\to\RR$ that satisfies~${F_j=f}$ on $J_j^\ve$	 and~$F_j'=f'(x_j^\ve)$ on $\RR\setminus J_j^\ve.$ 
Given $\xi\in(x_j^\ve-\pi,x_j^\ve+\pi)$, we then  have
\begin{equation}\label{gf2}
\begin{aligned}
&T_j^\ve(f)[\tau_y (\pi_j^\ve \varphi)-\pi_j^\ve \varphi](\xi)\\
&=\chi_j^\ve(\xi)\frac{1}{\pi}\PV\int\limits_{-\pi}^{\pi}\phi \Big(\frac{\dg{\xi,s}{f}}{s}\Big) \frac{\dg{\xi,s}{(f-f'(x_j^\ve)\id_\RR)}}{s} 
\frac{\big(\tau_y (\pi_j^\ve \varphi)-\pi_j^\ve \varphi\big)(\xi-s)}{s}\dx{s}\\
&=\chi_j^\ve(\xi)\frac{1}{\pi}\PV\int\limits_{-\pi}^{\pi}\phi \Big(\frac{\dg{\xi,s}{f}}{s}\Big)\frac{\dg{\xi,s}{(F_j-f'(x_j^\ve)\id_\RR)}}{s}
\frac{\big(\tau_y (\pi_j^\ve \varphi)-\pi_j^\ve \varphi\big)(\xi-s)}{s}\dx{s}\\
&=\big(\chi_j^\ve C_{n+1,m}(f,\dots,f)[f,\dots,f,F_j-f'(x_j^\ve)\id_\RR,\tau_y (\pi_j^\ve \varphi)-\pi_j^\ve \varphi]\big)(\xi),
\end{aligned}
\end{equation}
where $\phi(x)=x^n(1+x^2)^{-m},$ $x\in\RR.$
Indeed, if on the one hand $\xi\in (x_j^\ve-\pi,x_j^\ve+\pi)\setminus J_j^\ve,$ this is a consequence of $\chi_j^\ve(\xi)=0.$
If on the other hand $\xi\in J_j^\ve$, then $f(\xi)=F_j(\xi)$ by the definition of $F_j$. 
Moreover, since~${|s|<\pi}$, for sufficiently small $\ve$  we have $\xi-s\in (x_j^\ve-3\pi/2,x_j^\ve+3\pi/2),$ while
$\supp\pi_j^\ve\cap  (x_j^\ve-3\pi/2,x_j^\ve+3\pi/2)=I_j^\ve.$
Therefore, for~${\xi-s\not\in J_j^\ve}$ it holds $\xi-s+y\not\in I_j^\ve$ for all~${|y|\leq\ve}$, hence~${\pi_j^\ve \varphi(\xi-s)=\tau_y (\pi_j^\ve \varphi)(\xi-s)=0}$. 
Consequently, the integrand is not zero at most when $\xi-s\in J_j^\ve$, and in this case we also have   $f(\xi-s)=F_j(\xi-s).$ This proves~\eqref{gf2}.
 
Lemma~\ref{Lem:Cnm_est}~(i) together with \eqref{gf1}, \eqref{gf2}, and the definition of $F_j$  enables us to deduce that there exists a constant~$C>0$ such that for 
all $1\leq j\leq N,$ $|y|\leq \ve$, and ~${\varphi\in\Lp{2}{\Ss}}$ we have
\begin{align*}
\|T_j^\ve(f)[\tau_y (\pi_j^\ve \varphi)-\pi_j^\ve \varphi]\|_2\leq C\|f'-f'(x_j^\ve)\|_{\Lp{\infty}{J_j^\ve}}\|\tau_y (\pi_j^\ve \varphi)-\pi_j^\ve \varphi\|_2.
\end{align*}
The   estimate \eqref{eq:chiC} follows by choosing $\ve\in(0,1)$ sufficiently small in view of~${f'\in \rmC^{r-3/2}(\Ss).}$
	\end{proof}
	
	We are now in a position to establish Lemma \ref{Lem:Cnm_approx_a}.

	\begin{proof}[Proof of Lemma~\ref{Lem:Cnm_approx_a}]
In the following, we denote constants that do not depend on $\ve$ by~$C$ and constants that depend on $\ve$ by~$K$. 

Recalling that $ H=B_{0,0}^{0,0}$, cf. \eqref{eq:HT},  the relation~$ H=A^{1,0}_{0,0}+C_{0,0}$, cf. \eqref{eq:B=A+C}, together with Lemma~\ref{Lem:Anmq_Bnmpq_inf} yields
\[
\|( H-C_{0,0})[\pi_j^\ve\varphi]\|_{\nHp{r-1}}\leq C\|A^{1,0}_{0,0}[\pi_j^\ve\varphi]\|_{\rmC^1}\leq C\|\pi_j^\ve\varphi\|_\infty\leq K\norm{\varphi}_{\nHp{r'-1}},
\] 
and therefore
		\begin{equation*}
		\begin{aligned}
			&\bigg\|\pi_j^\ve a C_{n,m}^0 (f)[b\varphi]-\frac{a(x_j^\ve)b(x_j^\ve)(f'(x_j^\ve))^n}{\big[1+(f'(x_j^\ve))^2\big]^m}H[\pi_j^\ve \varphi]\bigg\|_{\nHp{r-1}}\\
			&  \leq \bigg\|\pi_j^\ve a C_{n,m}^0 (f)[b\varphi]-\frac{a(x_j^\ve)b(x_j^\ve)(f'(x_j^\ve))^n}{\big[1+(f'(x_j^\ve))^2\big]^m}C_{0,0}[\pi_j^\ve \varphi]\bigg\|_{\nHp{r-1}}+K\norm{\varphi}_{\nHp{r'-1}}.	
			\end{aligned}
		\end{equation*}
	To estimate the first term on the left-hand side of the latter inequality we write
		\begin{equation*}
			\pi_j^\ve a C_{n,m}^0 (f)[b\varphi]-\frac{a(x_j^\ve)b(x_j^\ve)(f'(x_j^\ve))^n}{\big[1+(f'(x_j^\ve))^2\big]^m}C_{0,0}[\pi_j^\ve \varphi]=a(T_1+T_2)+b(x_j^\ve)(T_3+ a(x_j^\ve)T_4),
		\end{equation*}
where
		\begin{equation*}
		\begin{aligned}
			T_1 &:= \pi_j^\ve C_{n,m}^0 (f)[(b-b(x_j^\ve))\varphi]-C_{n,m}^0(f)[\pi_j^\ve(b-b(x_j^\ve))\varphi],\\
			T_2 &:= C_{n,m}^0(f)[\pi_j^\ve(b-b(x_j^\ve))\varphi],\\
			T_3 &:= \pi_j^\ve a C_{n,m}^0(f)[\varphi]-a(x_j^\ve)C_{n,m}^0 (f)[\pi_j^\ve \varphi],\\
			T_4 &:= C_{n,m}^0(f)[\pi_j^\ve\varphi]-\frac{(f'(x_j^\ve))^n}{\big[1+(f'(x_j^\ve))^2\big]^m}C_{0,0}[\pi_j^\ve \varphi].
		\end{aligned}
		\end{equation*}
		We consider these terms successively.\medskip
		
		\noindent{\bf The term  $aT_1$.} In view of  Lemma~\ref{Lem:Cnm_com} and of the algebra property of $\Hp{r-1}{\Ss}$,  we have
		\begin{equation}\label{eq:T1_est}
			\norm{aT_1}_{\nHp{r-1}}\leq K\norm{(b-b(x_j^\ve))\varphi}_2\leq K\norm{\varphi}_{\nHp{r'-1}}.
		\end{equation}		
				
	\noindent{\bf The term  $aT_2$.} We use Lemma~\ref{Lem:Cnm_est}~(iii), \eqref{eq:Hr_BAlg}, the identity $\chi_j^\ve \pi_j^\ve =\pi_j^\ve$,  and the algebra property of $\Hp{r-1}{\Ss}$ to obtain,
	 in view of~$b\in\rmC^{r-3/2}(\Ss),$ that
		\begin{equation}\label{eq:T2_est}
		\begin{aligned}
			\norm{aT_2}_{\nHp{r-1}}&\leq C\norm{\pi_j^\ve(b-b(x_j^\ve))\varphi}_{\nHp{r-1}}\leq C\norm{\chi_j^\ve(b-b(x_j^\ve))}_\infty \norm{\pi_j^\ve\varphi}_{\nHp{r-1}}
			+K\norm{\varphi}_{\nHp{r'-1}}\\
			&\leq (\eta/3)\norm{\pi_j^\ve\varphi}_{\nHp{r-1}}+K\norm{\varphi}_{\nHp{r'-1}},
		\end{aligned}		
		\end{equation}
		provided that $\ve\in(0,1)$ is sufficiently small.\medskip
		
\noindent{\bf The term  $b(x_j^\ve)T_3$.}  Since $\chi_j^\ve \pi_j^\ve =\pi_j^\ve$, we have $T_3=T_{3,1}+T_{3,2}+T_{3,3}$, where
		\begin{equation*}
		\begin{aligned}
			T_{3,1}&:=(\chi_j^\ve a)\big(\pi_j^\ve C_{n,m}^{0}(f)[\varphi]- C_{n,m}^{0}(f)[\pi_j^\ve\varphi]\big),\\
			T_{3,2}&:=\chi_j^\ve (a-a(x_j^\ve)) C_{n,m}^{0}(f)[\pi_j^\ve \varphi],\\
			T_{3,3}&:=a(x_j^\ve)\big(\chi_j^\ve C_{n,m}^0(f)[\pi_j^\ve\varphi]-C_{n,m}^0(f)[\chi_j^\ve(\pi_j^\ve\varphi)]\big),
		\end{aligned}
		\end{equation*}
		and Lemma~\ref{Lem:Cnm_com} yields
		\begin{equation*}
			\norm{b(x_j^\ve)T_{3,1}}_{\nHp{r-1}}+\norm{b(x_j^\ve)T_{3,3}}_{\nHp{r-1}}\leq K\norm{\varphi}_{\nHp{r'-1}}.
		\end{equation*}
		Moreover,  \eqref{eq:Hr_BAlg}, Lemma~\ref{Lem:Cnm_est}~(iii),  and the property $a\in\rmC^{r-3/2}(\Ss)$ lead us to
		\begin{equation*} 
		\begin{aligned}
			\norm{b(x_j^\ve)T_{3,2}}_{\nHp{r-1}}&\leq C\norm{\chi_j^\ve(a-a(x_j^\ve))}_\infty\|C_{n,m}^{0}(f)[\pi_j^\ve \varphi]\|_{\nHp{r-1}}+K\|C_{n,m}^{0}(f)[\pi_j^\ve \varphi]\|_{\nHp{r'-1}}\\
			&\leq(\eta/3)\norm{\pi_j^\ve\varphi}_{\nHp{r-1}}+K\norm{\varphi}_{\nHp{r'-1}},
		\end{aligned}
		\end{equation*}
		provided that $\ve\in(0,1)$ is small enough, and therefore  
		\begin{equation}\label{eq:T3_est}
			\norm{b(x_j^\ve)T_3}_{\nHp{r-1}}\leq (\eta/3)\norm{\pi_j^\ve\varphi}_{\nHp{r-1}}+K\norm{\varphi}_{\nHp{r'-1}}.
		\end{equation}
		
		\noindent{\bf The term  $(ab)(x_j^\ve)T_4$.} Using again the relation $\chi_j^\ve \pi_j^\ve =\pi_j^\ve$, we have $T_4=T_{4,1}+T_{4,2}$, where
		\begin{equation*}
		\begin{aligned}
			T_{4,1}&:= \frac{(f'(x_j^\ve))^n}{\big[1+(f'(x_j^\ve))^2\big]^m}\big(\chi_j^\ve C_{0,0}[\pi_j^\ve \varphi]-C_{0,0}[\chi_j^\ve(\pi_j^\ve \varphi)]\big)\\
			&\qquad-\big(\chi_j^\ve C^{0}_{n,m}(f)[\pi_j^\ve \varphi]-C^{0}_{n,m}(f)[\chi_j^\ve (\pi_j^\ve \varphi)]\big),\\
			T_{4,2}&:=  \chi_j^\ve \bigg(C_{n,m}^{0}(f)[\pi_j^\ve \varphi]- \frac{(f'(x_j^\ve))^n}{\big[1+(f'(x_j^\ve))^2\big]^m} C_{0,0}[\pi_j^\ve \varphi]\bigg),
		\end{aligned}
		\end{equation*}
		and, by Lemma~\ref{Lem:Cnm_com},
		\begin{equation}\label{eq:T4_1_est}
			\norm{T_{4,1}}_{\nHp{r-1}}\leq K\norm{\varphi}_{2}.
		\end{equation}
		It remains to estimate the term $T_{4,2}$ for which we  first use Lemma~\ref{Lem:Cnm_est}~(i) to deduce that
		\begin{equation}\label{eq:T4_2_est}
			\norm{T_{4,2}}_2\leq K\norm{\varphi}_2.
		\end{equation}
		In order to estimate the seminorm $[T_{4,2}]_{\nWp{r-1}}$, we note,  by using~\eqref{eq:Cnm_dif_a}
		 together with the identity $f'(x_j^\ve)=\delta_{[\xi,s]}(f'(x_j^\ve)\id_\RR)/s$, that
		\begin{equation*} 
		\begin{aligned}
			T_{4,2} &= \sum_{k=0}^{n-1} (f'(x_j^\ve))^{n-k-1}\chi_j^\ve C_{k+1,m}(f,\dots,f)[f,\dots,f,f-f'(x_j^\ve)\id_\RR,\pi_j^\ve \varphi]\\
			&\quad-\sum_{k=0}^{m-1}\frac{(f'(x_j^\ve))^n}{\big[1+(f'(x_j^\ve))^2\big]^{m-k}}\chi_j^\ve C_{2,k+1}(f,\dots,f)[f,f-f'(x_j^\ve)\id_\RR,\pi_j^\ve \varphi]\\
			&\quad-\sum_{k=0}^{m-1}\frac{(f'(x_j^\ve))^{n+1}}{\big[1+(f'(x_j^\ve))^2\big]^{m-k}}\chi_j^\ve C_{1,k+1}(f,\dots,f)[f-f'(x_j^\ve)\id_\RR,\pi_j^\ve \varphi].
		\end{aligned}
		\end{equation*}
	Consequently,
	\begin{equation}\label{eq:T4_2_dec}
		\begin{aligned}
			[T_{4,2}]_{\nWp{r-1}} &\leq C_0\bigg( \sum_{k=0}^{n-1} \big[\chi_j^\ve C_{k+1,m}(f,\dots,f)[f,\dots,f,f-f'(x_j^\ve)\id_\RR,\pi_j^\ve \varphi]\big]_{\nWp{r-1}}\\
			&\hspace{1.25cm}+\sum_{k=0}^{m-1}\big[\chi_j^\ve C_{2,k+1}(f,\dots,f)[f,f-f'(x_j^\ve)\id_\RR,\pi_j^\ve \varphi]\big]_{\nWp{r-1}}\\
			&\hspace{1.25cm}+\sum_{k=0}^{m-1}\big[\chi_j^\ve C_{1,k+1}(f,\dots,f)[f-f'(x_j^\ve)\id_\RR,\pi_j^\ve \varphi]\big]_{\nWp{r-1}}\bigg).
		\end{aligned}
		\end{equation}
		Set
		\begin{equation*}
			S_k:= \chi_j^\ve C_{k+1,m}(f,\dots,f)[f,\dots,f,f-f'(x_j^\ve)\id_\RR,\pi_j^\ve \varphi],\qquad 0\leq k \leq n-1.
		\end{equation*}
		In order to estimate the $\nWp{r-1}$-seminorm of $S_k,$ we write for $ y\in(-\pi,\pi)$ 
		\begin{equation*}
			 \tau_y S_k-S_k =S_{k,1}+S_{k,2}+\chi_j^\ve S_{k,3},
		\end{equation*}
		where, using again \eqref{eq:Cnm_dif_a}, we have
		\begin{equation*}
		\begin{aligned}
			S_{k,1}&:=  (\tau_y \chi_j^\ve-\chi_j^\ve)\tau_y C_{k+1,m}(f,\dots,f)[f,\dots,f,f-f'(x_j^\ve)\id_\RR,\pi_j^\ve \varphi],\\
			S_{k,2}&:=  \chi_j^\ve C_{k+1,m}(f,\dots,f)[f,\dots,f,f-f'(x_j^\ve)\id_\RR,\tau_y(\pi_j^\ve \varphi)-\pi_j^\ve \varphi],\\
			S_{k,3}&:=  \sum_{i=1}^{k}C_{k+1,m}(f,\dots,f)[\underbrace{ f,\dots, f}_{i-1}, \tau_y f-f,\tau_y f,\dots,\tau_y f,f-f'(x_j^\ve)\id_\RR,\tau_y(\pi_j^\ve \varphi)]\\
			&\,\quad+C_{k+1,m}(f,\dots,f)[\tau_y f,\dots,\tau_y f,  \tau_y f-f,\tau_y(\pi_j^\ve \varphi)]\\
			&\,\quad-\sum_{i=1}^{m}C_{k+3,m+1}^i[\tau_y f,\dots, \tau_y f,\tau_y f-f'(x_j^\ve)\id_\RR, \tau_y f+f,\tau_y f-f, \tau_y (\pi_j^\ve \varphi)],
		\end{aligned}
		\end{equation*}
		and
		\begin{equation*}
			C_{k+3,m+1}^i:=C_{k+3,m+1}(\underbrace{f,\dots,f}_{i},\tau_y f,\dots, \tau_y f).
		\end{equation*}
		  Lemma~\ref{Lem:Cnm_est}~(iii)  (with $r=r'$) yields
		\begin{equation*}
			\norm{S_{k,1}}_2\leq K\norm{\tau_y \chi_j^\ve-\chi_j^\ve}_2 \norm{\varphi}_{\nHp{r'-1}}.
		\end{equation*}
		To estimate $S_{k,2}$ we consider two cases. If $|y|> \ve$, we use Lemma~\ref{Lem:Cnm_est}~(i) and obtain
		\begin{equation*}
			\norm{S_{k,2}}_2\leq K\norm{\varphi}_2.
		\end{equation*}		 
		If $|y|\leq\ve$, we use \eqref{eq:chiC} which  gives
		\begin{equation*}
			\norm{S_{k,2}}_2\leq (\eta/C_1) \norm{\tau_y(\pi_j^\ve \varphi)-\pi_j^\ve \varphi}_2,
		\end{equation*}
		provided that $\ve\in(0,1)$ is small enough, with a positive constant $C_1$ which we fix below. 
		Finally, Lemma~\ref{Lem:Cnm_est}~(ii) (with $r=r'$)   leads us to
		\begin{equation*}
			\norm{\chi_j^\ve S_{k,3}}_2\leq K\norm{\tau_y f'-f'}_2 \norm{\varphi}_{\nHp{r'-1}}.
		\end{equation*}
		Combining the above estimates, we have
		\begin{equation}\label{eq:Tk}
			[S_k]_{\nWp{r-1}}\leq  (\eta/C_1) \norm{\pi_j^\ve \varphi}_{\nHp{r-1}}+K\norm{\varphi}_{\nHp{r'-1}}.
		\end{equation}
		It is now obvious that actually all the terms on the right-hand side of \eqref{eq:T4_2_dec} can be estimated by the right-hand side of \eqref{eq:Tk}, provided that $\ve\in(0,1)$ is sufficiently small.
		From the estimates~\eqref{eq:T4_2_est}-\eqref{eq:Tk} we then deduce, after choosing $C_1:=3CC_0(n+2m)(1+\|ab\|_\infty),$ that 
		\begin{equation*}
		\begin{aligned}
		\norm{T_{4,2}}_{\nHp{r-1}}&\leq C(\| T_{4,2}\|_{2}+[ T_{4,2}]_{\nWp{r-1}})\leq \frac{CC_0(n+2m)\eta}{C_1}\norm{\pi_j^\ve \varphi}_{\nHp{r-1}}+K\norm{\varphi}_{\nHp{r'-1}}\\
		&\leq \frac{\eta}{3(1+\|ab\|_\infty)}\norm{\pi_j^\ve \varphi}_{\nHp{r-1}}+K\norm{\varphi}_{\nHp{r'-1}},
		\end{aligned}	
		\end{equation*}
		and together with \eqref{eq:T4_1_est} we get
		\begin{equation}\label{eq:T4_est}
			\norm{(ab)(x_j^\ve)T_4}_{\nHp{r-1}}\leq (\eta/3)\norm{\pi_j^\ve\varphi}_{\nHp{r-1}}+K\norm{\varphi}_{\nHp{r'-1}}.
		\end{equation}
		
 Gathering \eqref{eq:T1_est}-\eqref{eq:T3_est} and \eqref{eq:T4_est}, we obtain \eqref{eq:Cnm_approx_a}, and the proof is complete.
		\end{proof}

	\section{The behavior of the pressure and velocity near the interface and in the far-field}\label{Sec:C}	
	
	In this section we consider the function $(v^\pm,q^\pm)$ defined in \eqref{eq:vq}-\eqref{eq:vqdef} and prove, under the assumptions in Theorem~\ref{Thm:FT_unique}, 
	 that $(v^\pm,q^\pm)$ satisfies the boundary conditions~\mbox{\eqref{eq:refStokes}$_3$-\eqref{eq:refStokes}$_4$}, 
	 as well as the far field boundary condition~\eqref{eq:refStokes}$_5$, see Lemma~\ref{Lem:q_bd} and Lemma~\ref{Lem:v_bd} below.
	
	  Thus, in this section  we fix $f\in\Hp{3}{\Ss}$ and use the notation introduced in Section~\ref{Sec:FTP}. Additionally,  for  a given function $w:(\Ss\times\RR)\setminus\Gamma\to\RR$, 
	  we set $w^\pm:=w\big|_{\Omega^\pm}$ and  denote by
		\begin{equation*}
			\set{w}^\pm\circ\Xi(\xi):= \lim_{\Omega^{\pm} \ni x\to (\xi,f(\xi))}w(x),\quad \xi\in \Ss,
		\end{equation*}
		 the one-sided limits of  $w $ in $\Xi(\xi)$, whenever these limits exist. 
		 Moreover, given $w^\pm:\Omega^\pm\to\RR$, we set $w:={\bf 1}_{\Omega^+}w^++{\bf 1}_{\Omega^-}w^-$, which is viewed as a function defined a.e. in $\Ss\times\RR.$

In order to investigate  the gradient $\nabla v^\pm$, which  is determined by simply differentiating under the integral sign in~\eqref{eq:vqdef} (see the proof of Theorem~\ref{Thm:FT_unique}),
 we infer from \eqref{eq:UPdef} that, for given  $ x\in(\Ss\times\RR)\setminus\{0\}$,  the first order partial derivatives of $\mathcal{U}$  are given by  
		\begin{equation}\label{eq:deriv_U}
		\begin{aligned}		
		\partial_{1}{\mcU^{1}}^\top(x)&=\frac{1}{8\pi}
		\begin{pmatrix}
		\frac{\tx(1-\Tx^2)}{\tx^2+\Tx^2}-x_2\frac{\tx\Tx(1+\tx^2)(1-\Tx^2)}{(\tx^2+\Tx^2)^2}\\[2ex]
		\frac{x_2}{2}\frac{(1+\tx^2)(1-\Tx^2)(\tx^2-\Tx^2)}{(\tx^2+\Tx^2)^2}
		\end{pmatrix},\\
		\partial_{2}{\mcU^{1}}^\top(x)&=\frac{1}{8\pi}
		\begin{pmatrix}
		2\frac{\Tx(1+\tx^2)}{\tx^2+\Tx^2}+\frac{x_2}{2}\frac{(1+\tx^2)(1-\Tx^2)(\tx^2-\Tx^2)}{(\tx^2+\Tx^2)^2}\\[2ex]
		-\frac{\tx(1-\Tx^2)}{\tx^2+\Tx^2}+x_2\frac{\tx\Tx(1+\tx^2)(1-\Tx^2)}{(\tx^2+\Tx^2)^2}	
		\end{pmatrix},\\
		\partial_{1}{\mcU^{2}}^\top(x)&=\frac{1}{8\pi}
		\begin{pmatrix}
		\frac{x_2}{2}\frac{(1+\tx^2)(1-\Tx^2)(\tx^2-\Tx^2)}{(\tx^2+\Tx^2)^2}\\[2ex]
		\frac{\tx(1-\Tx^2)}{\tx^2+\Tx^2}+x_2\frac{\tx\Tx(1+\tx^2)(1-\Tx^2)}{(\tx^2+\Tx^2)^2}
		\end{pmatrix},\\
		\partial_{2}{\mcU^{2}}^\top(x)&=\frac{1}{8\pi}
		\begin{pmatrix}
		-\frac{\tx(1-\Tx^2)}{\tx^2+\Tx^2}+x_2\frac{\tx\Tx(1+\tx^2)(1-\Tx^2)}{(\tx^2+\Tx^2)^2}\\[2ex]
		-\frac{x_2}{2}\frac{(1+\tx^2)(1-\Tx^2)(\tx^2-\Tx^2)}{(\tx^2+\Tx^2)^2}
		\end{pmatrix}.\\
		\end{aligned}
		\end{equation}

This motivates us to establish first the following preparatory result. 
	
\begin{Lemma}\label{Lem:C1} Given~${\varphi \in\Lp{2}{\Ss}}$,   let $	Z_n(f)[\varphi]:(\Ss\times\RR)\setminus\Gamma\to\RR$, $ 1\leq n\leq 6,$ 
be defined by
\begin{subequations}\label{eq:defZ}
\begin{equation}
		\begin{aligned}
			Z_1(f)[\varphi](x) &:= \frac{1}{2\pi}\int_{-\pi}^\pi\frac{\ts{r_1}(1-\T{r_2}^2)}{\ts{r_1}^2+\T{r_2}^2}\varphi(s)\dx{s},\\
			Z_2(f)[\varphi](x) &:= \frac{1}{2\pi}\int_{-\pi}^\pi\frac{\T{r_2}(1+\ts{r_1}^2)}{\ts{r_1}^2+\T{r_2}^2}\varphi(s)\dx{s},
		\end{aligned}
		\end{equation}
		\begin{equation}
		\begin{aligned}
			Z_3(f)[\varphi](x) &:= \frac{1}{2\pi}\int_{-\pi}^\pi\frac{r_2}{2}\frac{(1+\ts{r_1}^2)(1-\T{r_2}^2)(\ts{r_1}^2-\T{r_2}^2)}{(\ts{r_1}^2+\T{r_2}^2)^2}\varphi(s)\dx{s},\\
			Z_4(f)[\varphi](x) &:= \frac{1}{2\pi}\int_{-\pi}^\pi\frac{r_2}{2}\frac{\ts{r_1}\T{r_2}(1+\ts{r_1}^2)(1-\T{r_2}^2)}{(\ts{r_1}^2+\T{r_2}^2)^2}\varphi(s)\dx{s},\\
			Z_5(f)[\varphi](x) &:=\frac{1}{2\pi} \int_{-\pi}^\pi r_2\frac{\ts{r_1}(1-\T{r_2}^2)}{\ts{r_1}^2+\T{r_2}^2}\varphi(s)\dx{s},\\
			Z_6(f)[\varphi](x) &:= \frac{1}{2\pi}\int_{-\pi}^\pi r_2\frac{\T{r_2}(1+\ts{r_1}^2)}{\ts{r_1}^2+\T{r_2}^2}\varphi(s)\dx{s},
		\end{aligned}
		\end{equation}
\end{subequations}
		where  $r=(r_1,r_2)$ is defined by
			\begin{equation}\label{eq:rM}
 r:=r(x,s):=x-(s,f(s)),\qquad x\in\Omega^\pm,\,s\in\RR.
	\end{equation}
We set
	\begin{equation}\label{eq:def_B}
			B_{n}(f)[\varphi](\xi):= \PV\,\,\big(Z_{n}(f)[\varphi](\Xi(\xi))\big),\qquad 1\le n\le 6,\quad \xi\in\Ss.
		\end{equation}
	Then, $Z_n(f)[\varphi]^\pm\in{\rm C}^\infty(\Omega^\pm)$, $1\leq n\leq 6$, and     $Z_n(f)[\varphi] \in {\rm C}(\Ss\times\RR)$, $n=5,\, 6$,  with
		\begin{equation}\label{eq:Z_pm1}
			\big\{ Z_{n}(f)[\varphi] \big\}^\pm\circ\Xi =  B_{n}(f)[\varphi],\qquad  n=5,\, 6.
		\end{equation}
	Moreover, if additionally $\varphi \in \Hp{1}{\Ss}$, then  $Z_n(f)[\varphi]^\pm \in {\rm C}(\overline{\Omega^\pm})$, $1\leq n\leq 4,$
	and
		\begin{equation}\label{eq:Z_pm2}
			\left\{\begin{pmatrix}
			Z_{1}(f)[\varphi]\\ Z_{2}(f)[\varphi]\\ Z_{3}(f)[\varphi]\\ Z_{4}(f)[\varphi] 
			\end{pmatrix} \right\}^\pm\circ\Xi = \begin{pmatrix}
			B_{1}(f)[\varphi]\\ B_{2}(f)[\varphi]\\ B_{3}(f)[\varphi]\\ B_{4}(f)[\varphi] 
			\end{pmatrix}\pm   \frac{1}{\omega^2}
			\begin{pmatrix}
			-f' \\1 \\-\frac{2f'^2}{\omega^2}\\ \frac{f'-f'^3}{2\omega^2} 
			\end{pmatrix} \varphi.
		\end{equation}
\end{Lemma}	

Related to the definition \eqref{eq:def_B}, we note that  we may evaluate the integrals  \eqref{eq:defZ} at~$\Xi(\xi)$ with~$\xi\in\Ss$, 
provided that we interpret some of  the integrals as being  singular,  see
		 Lemma~\ref{Lem:Anmq_Bnmpq_inf} and Lemma~\ref{Lem:Bnmq_L2_L2}, 
		 since the operators $B_n(f)$, $1\leq n\leq 6$, can be represented as linear combinations of the operators 
		$B_{n,m}^{p,q}(f)$, $n,\, m,\,p,\,\,q\in\NN_0$, ${1\leq p\leq n+q+1},$ defined in \eqref{eq:B(f)}, see \eqref{eq:B_by_Bnmpq}.
		In fact, Lemma~\ref{Lem:Anmq_Bnmpq_inf} ensures that  $ B_n(f)[\varphi]\in{\rm C}(\Ss)$, $n=5,\, 6$, while, for $\varphi \in \Hp{1}{\Ss}$,
 we also have~${ B_n(f)[\varphi]\in{\rm C}(\Ss)}$, $1\leq n\leq 4$, cf.  Lemma~\ref{Lem:Anmq_Bnmpq_inf} and Lemma~\ref{Lem:Bnmq_Hr-1_Hr-1}. 
	
	\begin{proof}[Proof of Lemma~\ref{Lem:C1}]
	Arguing as in the proof of Theorem~\ref{Thm:FT_unique}, it immediately follows that the function~$Z_n(f)[\varphi]^\pm$ belongs to $ {\rm C}^\infty(\Omega^\pm)$  for $1\leq n\leq 6$.
	Moreover,	 Lebesgue's dominated convergence   theorem  leads to
		\begin{equation*}
			 \{Z_n(f)[\varphi] \}^\pm\circ\Xi = B_n(f)[\varphi]\in {\rm C}(\Ss),\qquad n=5,\, 6,
		\end{equation*}
		so that $Z_n(f)[\varphi] \in {\rm C}(\Ss\times\RR)$ for $n=5,\, 6$. This proves \eqref{eq:Z_pm1}.
	
	In the remaining we assume that   $\varphi \in \Hp{1}{\Ss}$. 
	Since~${B_n(f)\in{\rm C}(\Ss)}$, $n=1,\,2$,   together with~\cite[Lemma~2.2]{Matioc.2020}, we  conclude   that  
	 $Z_{n}(f)[\varphi]^\pm\in {\rm C}(\overline{\Omega^\pm})$ for~${n=1,\, 2}$,  with 
		\begin{equation}\label{eq:lin1}
		\begin{aligned}
			\set{Z_{1}(f)[\varphi]}^\pm\circ\Xi &= B_{1}(f)[\varphi]\mp \frac{f'}{\omega^2}\varphi,\\
			\set{Z_{2}(f)[\varphi]}^\pm \circ\Xi  &= B_{2}(f)[\varphi]\pm \frac{1}{\omega^2}\varphi.
		\end{aligned}
		\end{equation}
	In order to derive similar properties for $Z_{n}(f)[\varphi]$,~${n=3,\, 4,} $ we use integration by parts to  deduce that
		\begin{equation*}
\left.
\begin{array}{lll}
Z_5(f)[\varphi']&=Z_{1}(f)[f'\varphi]-Z_{3}(f)[\varphi]-2Z_{4}(f)[f'\varphi],\\[1ex]
Z_6(f)[\varphi']&=Z_{2}(f)[f'\varphi]+Z_{3}(f)[f'\varphi]-2Z_{4}(f)[\varphi]
\end{array}
\right\}	\qquad\text{in $(\Ss\times\RR)\setminus\Gamma,$}	
		\end{equation*}
		 respectively
		\begin{equation*}
		\left.
\begin{array}{lll}
B_5(f)[\varphi']&=B_{1}(f)[f'\varphi]-B_{3}(f)[\varphi]-2B_{4}(f)[f'\varphi],\\[1ex]
B_6(f)[\varphi']&=B_{2}(f)[f'\varphi]+B_{3}(f)[f'\varphi]-2B_{4}(f)[\varphi]
\end{array}
\right\}	\qquad\text{in ${\rm C}(\Ss).$}		
		\end{equation*}
		Since  $Z_n(f)[\varphi'] \in {\rm C}(\Ss\times\RR),$   $n=5,\, 6,$  the latter formulas combined with \eqref{eq:Z_pm1} and~\eqref{eq:lin1} (with~$\varphi$ replaced by $f'\varphi$) yield
		\begin{equation}\label{eq:lin2}
		\begin{aligned}
			\set{Z_{3}(f)[\varphi]+2Z_{4}(f)[f'\varphi]}^\pm\circ\Xi &= B_{3}(f)[\varphi]+2B_{4}(f)[f'\varphi]\mp \frac{f'^2}{\omega^2}\varphi,\\
			\set{Z_{3}(f)[f'\varphi]-2Z_{4}(f)[\varphi]}^\pm\circ\Xi &= B_{3}(f)[f'\varphi]-2B_{4}(f)[\varphi]\mp \frac{f'}{\omega^2}\varphi.
		\end{aligned}
		\end{equation}	
	We now  replace $\varphi$ by $\varphi/\omega^2$ in \eqref{eq:lin2}$_1$ and by $(f'\varphi )/\omega^2$ in \eqref{eq:lin2}$_2$ to obtain, after taking the sum of the two relations,
	that 
		\begin{equation}\label{eq:lin3}
		\begin{aligned}
			\set{Z_{3}(f)[\varphi]}^\pm\circ\Xi &= B_{3}(f)[\varphi]\mp \frac{2f'^2}{\omega^4}\varphi,\\
			\set{Z_{4}(f)[\varphi]}^\pm \circ\Xi  &= B_{4}(f)[\varphi]\pm \frac{f'-f'^3}{2\omega^4}\varphi,
		\end{aligned}
		\end{equation}
		with \eqref{eq:lin3}$_2$ being a direct consequence of \eqref{eq:lin3}$_1$  and \eqref{eq:lin2}$_2$.
		This proves \eqref{eq:Z_pm2} and completes the proof.
	\end{proof}

		As a further  preparatory result we establish the following lemma which is related to the logarithmic term in $\mathcal{U}$, see \eqref{eq:GPi} and \eqref{eq:UPdef}. 
	
\begin{Lemma}\label{Lem:C2} Given~${\varphi \in\Lp{2}{\Ss}}$,   let $	Z_0(f)[\varphi]:(\Ss\times\RR)\setminus\Gamma\to\RR$ 
be given by
\begin{equation}\label{eq:def_Z0}
			Z_0(f)[\varphi](x) := \frac{1}{2\pi}\int_{-\pi}^\pi \ln \Big(\sin^2\Big(\frac{r_1}{2}\Big)+\sinh^2\Big(\frac{r_2}{2}\Big)\Big)\varphi(s)\dx{s}.
		\end{equation}
Then, 	$Z_0(f)[\varphi]\in {\rm C}^\infty((\Ss\times\RR)\setminus\Gamma)$ and~${\nabla \big( Z_0(f)[\varphi]\big) =( Z_1(f)[\varphi],  Z_2(f)[\varphi])}$.
Additionally, if~${\varphi\in \Hp{1}{\Ss},}$  we have $Z_0(f)[\varphi]\in{\rm C}(\Ss\times\RR)$  and $Z_0(f)[\varphi]^\pm \in\rmC^{1}(\overline{\Omega^\pm},\RR^2)$, with 
\begin{equation}\label{eq:Z_0}
			\set{Z_{0}(f)[\varphi]}^\pm\circ\Xi = B_{0}(f)[\varphi],
		\end{equation}
		where $B_0(f)$ is defined in \eqref{eq:B0_alt} (see also \eqref{eq:GPi}).
\end{Lemma}	
\begin{proof}
Arguing as in the proof of Theorem~\ref{Thm:FT_unique}, we obtain that ~$Z_0(f)[\varphi]^\pm\in  {\rm C}^\infty(\Omega^\pm)$, with gradient~${\nabla \big( Z_0(f)[\varphi]\big) =( Z_1(f)[\varphi],  Z_2(f)[\varphi])}$.
Since  $Z_n(f)[\varphi]^\pm\in  {\rm C}^1(\overline{\Omega^\pm})$, $n=1,\, 2$, for~${\varphi\in \Hp{1}{\Ss}}$, cf. Lemma~\ref{Lem:C1}, we deduce that $Z_0(f)[\varphi]^\pm \in\rmC^{1}(\overline{\Omega^\pm},\RR^2)$.
Lebesgue's dominated convergence theorem  ensures moreover that both one-sided limits of $Z_0(f)[\varphi]$ in $\Xi(\xi)$  exist for all $\xi\in\Ss$ and coincide with $B_{0}(f)[\varphi](\xi)$.
 This proves~\eqref{eq:Z_0} and the continuity property $Z_0(f)[\varphi]\in{\rm C}(\Ss\times\RR)$.
\end{proof}

Related to the asymptotic behavior of the operators defined above,  we establish the following lemma.

\begin{Lemma}\label{Lem:C3} 
Given~${\varphi \in\Lp{2}{\Ss}}$,     for $ x_2 \to\pm\infty$ we have
\begin{align}
         &Z_5(f)[\varphi]^\pm\longrightarrow 0,\label{eq:AB5}\\[1ex]
		 &Z_6(f)[\varphi]^\pm\mp x_2\langle \varphi\rangle\longrightarrow \mp \langle f\varphi\rangle,\label{eq:AB6} \\[1ex]
		&Z_0(f)[\varphi]^\pm\mp x_2\langle \varphi\rangle\longrightarrow \mp\langle f\varphi\rangle-\langle \varphi\rangle\ln 4.\label{eq:AB0}
		\end{align}
\end{Lemma}	
\begin{proof} The property \eqref{eq:AB5} is a simple consequence of Lebesgue's dominated convergence theorem, which implies, via 
\begin{align*}
Z_6(f)[\varphi]^\pm(x)\mp x_2\langle \varphi\rangle\pm \langle f\varphi\rangle=\frac{1}{ 2\pi}\int_{-\pi}^\pi r_2(1\mp \T{r_2})\frac{\T{r_2}\mp\ts{r_1}^2}{\ts{r_1}^2+\T{r_2}^2} \varphi(s) \dx{s},
\qquad x\in\Omega^\pm,
\end{align*}
  also \eqref{eq:AB6}. 
  Finally, with respect to \eqref{eq:AB0}, we note that  since 
\[
-\frac{1}{ 2\pi}\int_{-\pi}^\pi \ln (4^{\mp1}e^{r_2})\varphi\dx{s}=\langle f\varphi\rangle+(\pm \ln 4-x_2)\langle \varphi\rangle,\qquad x\in\Omega^\pm,
\]
 Lebesgue's dominated convergence theorem yields
\begin{align*}
&Z_0(f)[\varphi]^\pm(x)\pm\big[\langle f\varphi\rangle+(\pm \ln 4-x_2)\langle \varphi\rangle\big]\\
&=\frac{1}{ 2\pi}\int_{-\pi}^\pi  \Big[\ln\Big(\sin^2\Big(\frac{r_1}{2}\Big)+\frac{e^{r_2}-2+e^{-r_2}}{4}\Big)\mp \ln (4^{\mp1}e^{r_2})\Big]\varphi(s)\dx{s}\\
&=\frac{1}{ 2\pi}\int_{-\pi}^\pi \ln\Big(4e^{\mp r_2}\sin^2\Big(\frac{r_1}{2}\Big)+ e^{\mp 2r_2 }-2e^{\mp r_2}+1\Big) \varphi(s)\dx{s}\underset{x_2\to\pm\infty}\longrightarrow0.\qedhere
\end{align*}
\end{proof}

We are now in a position to study the behavior of the velocity  $v$ defined in~\eqref{eq:vq}-\eqref{eq:vqdef} close to the interface and in the far field (under the assumptions of Theorem~\ref{Thm:FT_unique}).

	\begin{Lemma}\label{Lem:v_bd}
	We have  $v\in\mathrm{C}(\Ss\times\RR,\RR^2)$, $v^\pm\in\mathrm{C}^{1}(\overline{\Omega^\pm},\RR^2)$, and 
\begin{align}
&\mu\big([\nabla v+(\nabla v)^\top]\wt\nu\big)\circ\Xi=\omega^{-1}(G\cdot\tau)\tau \qquad\text{on $\Ss$},\label{eq:stressv}\\[1ex]
&\text{$v^\pm (x)\longrightarrow \Big(\mp \frac{\langle f G_1\rangle}{2\mu},0\Big)$ \qquad for $x_{2}\to \pm\infty$.}\label{eq:ascond}
\end{align}	
	\end{Lemma}
	\begin{proof} 	
Recalling \eqref{eq:vqdef}, we write
		\begin{equation}\label{eq:vG_z}
			v_G=\frac{1}{4\mu}
			\begin{pmatrix}
			\big(Z_0(f)+Z_6(f)\big)[G_1]-Z_5(f)[G_2]\\[1ex]
			\big(Z_0(f)-Z_6(f)\big)[G_2]-Z_5(f)[G_1]
			\end{pmatrix}^\top\qquad\text{in $(\Ss\times\RR)\setminus\Gamma,$}
	\end{equation}
	and Lemma \ref{Lem:C1} and Lemma~\ref{Lem:C2} ensure that indeed~$v_G\in\mathrm{C}(\Ss\times\RR,\RR^2)$ and
	 \begin{equation}\label{eq:v_g}
			\big\{v_G\big\}^\pm\circ\Xi=\frac{1}{4\mu}
			\begin{pmatrix}
			\big(B_0(f)+B_6(f)\big)[G_1]-B_5(f)[G_2]\\[1ex]
			\big(B_0(f)-B_6(f)\big)[G_2]-B_5(f)[G_1]
			\end{pmatrix}^\top.
	\end{equation}
	
Noticing also that
\begin{equation*}
\left.
\begin{array}{lll}
{\nabla \big( Z_5(f)[\varphi]\big) =( -Z_3(f)[\varphi],  Z_1(f)[\varphi]- 2Z_4(f)[\varphi])}\\
{\nabla \big( Z_6(f)[\varphi]\big) =( -2Z_4(f)[\varphi],  Z_2(f)[\varphi]+Z_3(f)[\varphi])}
\end{array}
\right\}\qquad\text{in $\Ss\times\RR\setminus\Gamma$,}
\end{equation*}
we infer from Lemma~\ref{Lem:C1} and Lemma~\ref{Lem:C2} that $v_G^\pm\in\rmC^{1}(\overline{\Omega^\pm},\RR^2)$ and the formula \eqref{eq:Z_pm2} leads us to
		 \begin{equation*}
		 [\nabla v_G]\circ\Xi=\begin{pmatrix}
		 [\partial_1 v_{G,1}]\circ\Xi&[\partial_2 v_{G,1}]\circ\Xi\\[1ex]
		 [\partial_1 v_{G,2}]\circ\Xi&[\partial_2 v_{G,2}]\circ\Xi
		 \end{pmatrix}
		 =\frac{G\cdot \tau}{\mu \omega^3}
		 \begin{pmatrix}
	-f'&1\\[1ex]
	-f'^2&f'
		 \end{pmatrix},
		 \end{equation*}
		 hence
		  \begin{equation*}
		\mu \big([\nabla v_G+(\nabla v_G)^\top]\wt\nu\big)\circ\Xi=\omega^{-1}(G\cdot\tau)\tau,
		 \end{equation*}
 and  \eqref{eq:stressv} follows.
 
 Moreover, in view of Lemma~\ref{Lem:C3}, we have
 \[
 v_G^\pm (x)\longrightarrow \Big(\mp \frac{\langle fG_1\rangle}{2\mu},- \frac{\langle G_2\rangle\ln 4}{4\mu}\Big) \qquad \text{for $x_{2}\to \pm\infty$,}
 \]
 which proves \eqref{eq:ascond}.
\end{proof}

	The following observation,  together with~\eqref{eq:v_g} is used when formulating the Stokes problem \eqref{eq:STOKES} as an evolution problem for $f$, cf. \eqref{eq:ev_eq},
	 as it provides an expression for the trace of $v_G$ on $\Gamma$,  in the particular case when~$G=F'$ for some function $F=(F_1,F_2)$, which involves the function $F$
	(and not its derivative), see \eqref{eq:v_z} below.
	
	\begin{Remark}\label{R:C1}
	 Assume that~$G=F'$ for some function $F=(F_1,F_2)\in\Hp{2}{\Ss}$.
Then, observing that ~${[s\mapsto \mcU(x-(s,f(s)))]:\Ss\to\RR^{2\times2}}$ is continuously differentiable, integration by parts  in  \eqref{eq:vqdef} leads to  the following representation
	\begin{equation}\label{eq:vqdef_alt}
		 {v_G^\pm}(x) =\frac{1}{\mu}\int_{-\pi}^\pi F(s)\left(\partial_{1}\begin{pmatrix}
			{\mcU^1}\\
			{\mcU^2}
		\end{pmatrix}(r)+f'(s)\partial_{2}\begin{pmatrix}
			{\mcU^1}\\
			{\mcU^2}
		\end{pmatrix}(r)\right)\dx{s}, \qquad x\in \Omega^\pm.
	\end{equation}
	 In view of   \eqref{eq:deriv_U} and  \eqref{eq:defZ}, we conclude from Lemma~\ref{Lem:C1} that 
		 \begin{equation}\label{eq:v_z}
			\big\{v_G\big\}^\pm\circ\Xi=\frac{1}{4\mu}
			\begin{pmatrix}
			\big(B_1-2B_4\big)(f)[F_1-f' F_2]+\big(2B_2+B_3)(f)[f' F_1]+B_3(f)[F_2]\\[1ex]
			 B_1(f)[F_2-f' F_1]+B_3(f)[F_1-f' F_2]+2B_4(f)[f' F_1+F_2],
			\end{pmatrix}^\top.
		\end{equation}
	\end{Remark}

	Finally, we  consider the pressure $q$.
	
	\begin{Lemma}\label{Lem:q_bd}
	We have $q^\pm\in\rmC(\overline{\Omega^\pm})$ and 
		\begin{align*}
			&[q]\circ\Xi=-\omega^{-1}G\cdot \nu \qquad\text{on $\Ss$},\\[1ex]
			&q^\pm (x)\longrightarrow  \mp\frac{\langle G_2\rangle}{2} \quad \text{for $x_{2}\to \pm\infty.$}
		\end{align*}
	\end{Lemma}
	\begin{proof}
		Since
		\[
q=-\frac{Z_1(f)[G_1]+Z_2(f)[G_2]}{2},		
		\]
		Lemma~\ref{Lem:C1} yields $q^\pm\in\rmC^1(\overline{\Omega^\pm})$ together with $[q]\circ\Xi=-\omega^{-1}G\cdot \nu.$
		Moreover, a simple application of Lebesgue's dominated convergence theorem shows that 
		\[
q^\pm (x)\longrightarrow \mp\frac{\langle G_2\rangle}{2} \qquad \text{for $x_{2}\to \pm\infty,$}		
		\]
		which completes the proof.
	\end{proof}

\bibliographystyle{siam}
\bibliography{references}
\end{document}